\documentclass[12pt, twoside]{article}

\def\titlerunning#1{\gdef\titrun{#1}}
\makeatletter
\def\author#1{\gdef\autrun{\def\and{\unskip, }#1}\gdef\@author{#1}}
\def\address#1{{\def\and{\\\hspace*{18pt}}\renewcommand{\thefootnote}{}%
\footnote {#1}}%
\markboth{\autrun}{\titrun}}
\makeatother
\def\email#1{e-mail: #1}
\def\subjclass#1{{\renewcommand{\thefootnote}{}%
\footnote{\emph{Mathematics Subject Classification (2010):} #1}}}

\usepackage{enumerate}
\usepackage{caption, subcaption}

\usepackage[utf8]{inputenc}
\usepackage[T1]{fontenc}
\usepackage[english]{babel}
\usepackage{enumitem}
\usepackage{amsmath,amsfonts,amssymb,graphicx,bbm,stmaryrd}

\usepackage{color}

\usepackage{amsthm}
\usepackage{constants}
\newconstantfamily{small}{
symbol=c}

\theoremstyle{plain}
\newtheorem{theorem}{Theorem}
\newtheorem{proposition}[theorem]{Proposition}
\newtheorem{lemma}[theorem]{Lemma}
\newtheorem{corollary}[theorem]{Corollary}

\theoremstyle{remark}
\newtheorem*{remark}{Remark}
\theoremstyle{definition}

\usepackage{mathtools}

\usepackage{dsfont}

\newcommand{\calE}{\mathcal{E}}
\newcommand{\calF}{\mathcal{F}}
\newcommand{\calG}{\mathcal{G}}

\newcommand{\calV}{\mathcal{V}}

\newcommand{\bbH}{\mathbb{H}}

\newcommand{\bbN}{\mathbb{N}}

\newcommand{\bbP}{\mathbb{P}}

\newcommand{\bbR}{\mathbb{R}}

\newcommand{\bbT}{\mathbb{T}}

\newcommand{\bbZ}{\mathbb{Z}}

\newcommand{\ep}{\varepsilon}

\newcommand{\T}{\mathbb T}

\newcommand{\R}{\mathbb{R}}

\newcommand{\wrel}{\textsc{w}}
\newcommand{\1}{\mathbbm{1}}

\newcommand{\wind}{\mathrm{wind}}

\usepackage{verbatim}

\usepackage{tabularx}

\mathtoolsset{showonlyrefs}
\usepackage{hyperref}

\begin{document}


\baselineskip=17pt


\titlerunning{Macroscopic loops in the loop $O(n)$ model }

\title{Macroscopic loops in the loop $O(n)$ model \\at Nienhuis' critical point}

\author{Hugo Duminil-Copin
\and
Alexander Glazman
\and
Ron Peled
\and
Yinon Spinka}

\date{\today}

\maketitle

\address{H. Duminil-Copin: Institut des Hautes \`Etudes Scientifiques, Bures-sur-Yvette, France, and
    Universit\'e de Gen\`eve, D\'epartement de math\'ematiques, Switzerland; \\ \indent\indent \email{duminil@ihes.fr}
\and
A. Glazman: Faculty of Mathematics, University of Vienna, Austria; \\ \indent\indent\email{alexander.glazman@univie.ac.at}
\and
R. Peled: School of Mathematical Sciences, Tel Aviv University, Israel; \\ \indent\indent \email{peledron@post.tau.ac.il}
\and
Y. Spinka: Department of Mathematics, University of British Columbia, Canada; \\ \indent\indent \email{yinon@math.ubc.ca}.}



\begin{abstract}
The loop $O(n)$ model is a model for a random collection of non-intersecting loops on
the hexagonal lattice, which is believed to be in the same
universality class as the spin $O(n)$ model. It has been predicted
by Nienhuis that for $0\le n\le 2$, the loop $O(n)$ model exhibits a phase transition at a critical parameter $x_c(n)=1/\sqrt{2+\sqrt{2-n}}$. For $0<n\le 2$, the transition line has been further conjectured to separate a regime with short loops when $x<x_c(n)$ from a regime with macroscopic loops when $x\ge x_c(n)$.

In this paper, we prove that for $n\in [1,2]$ and $x=x_c(n)$, the loop $O(n)$ model exhibits macroscopic loops. Apart from the case $n=1$, this constitutes the first regime of parameters for which macroscopic loops have been rigorously established. A main tool in the proof is a new positive association (FKG) property shown to hold when $n \ge 1$ and $0<x\le\frac{1}{\sqrt{n}}$. This property implies, using techniques recently developed for the random-cluster model, the following dichotomy: either long loops are exponentially unlikely or the origin is surrounded by loops at any scale (box-crossing property). We develop a `domain gluing' technique which allows us to employ Smirnov's parafermionic observable to rule out the first alternative when $n\in[1,2]$ and $x=x_c(n)$.
\end{abstract}

\section{Introduction}\label{sec:introduction}

\subsection{Historical background}
After the introduction of the Ising model \cite{Len20} and Ising's
conjecture that it does not undergo a phase transition, physicists
tried to find natural generalizations of the model with richer
behavior. In~\cite{HelKra34}, Heller and Kramers described the
classical version of the celebrated quantum Heisenberg model, where
spins are vectors in the (two-dimensional) unit sphere in dimension
three. In 1966, Vaks and Larkin introduced the XY model \cite{VakLar66}, and a few years later, Stanley proposed a more general model, called the \emph{spin $O(n)$ model},
allowing spins to take values in higher-dimensional spheres
\cite{Sta68}. We refer the interested reader to~\cite{Sta74} for
a history of the subject.
On the hexagonal lattice, the spin $O(n)$ model can be related
to the so-called \emph{loop $O(n)$ model} introduced
in~\cite{DomMukNie81} (see also \cite{DumPelSam14} for more details on this connection and~\cite{PelSpi17} for a survey).

More formally, the loop $O(n)$ model is defined as follows. Consider the triangular lattice $\bbT$ composed of vertices with complex coordinates $r+{\rm e}^{{\rm i}\pi/3}s$ with $r,s\in\bbZ$, and its dual lattice, the hexagonal lattice $\bbH$. Since $\bbT$ and $\bbH$ are dual to each other, we call vertices of $\bbT$ {\em hexagons} to highlight the fact that they are in correspondence with faces of $\bbH$.

A \emph{loop configuration} is a spanning subgraph of $\bbH$ in which
every vertex has even degree. Note that a loop configuration can a priori consist of loops (i.e., subgraphs which are isomorphic to a cycle) together with isolated vertices and infinite paths.
For a finite set of edges $\Omega$ of the hexagonal lattice $\bbH$ and a loop configuration $\xi$, let $\calE(\Omega,\xi)$ be the set of loop configurations coinciding with $\xi$ outside $\Omega$. Let $n$ and $x$ be
positive real numbers. The loop $O(n)$ measure on $\Omega$ with edge-weight $x$ and boundary conditions $\xi$ is the probability measure
$\bbP_{\Omega,n,x}^\xi$ on $\calE(\Omega,\xi)$ defined by the formula
  \[
  \bbP_{\Omega,n,x}^\xi(\omega) := \frac{x^{|\omega|} n^{\ell(\omega)}}{Z_{\Omega,n,x}^\xi},  \]
for every $\omega\in \calE(\Omega,\xi)$,  where $|\omega|$ is the number of edges of $\omega\cap\Omega$, $\ell(\omega)$ is the number of loops of $\omega$ intersecting $\Omega$, and $Z_{\Omega,n,x}^\xi$ is the unique constant making $\bbP_{\Omega,n,x}^\xi$ a probability measure.

The physics predictions on the loop $O(n)$ model are quite mesmerizing. Nienhuis conjectured \cite{Nie82,Nie84} the following behavior: for $n\le 2$ and $x$ strictly smaller than
\begin{equation}
x_c(n):=\frac{1}{\sqrt{2+\sqrt{2-n}}},
\end{equation}
the probability that a given vertex is on a long loop decays exponentially fast in the length of the loop (subcritical regime), while for $x\ge x_c(n)$ it decays as a power-law. For $n>2$, the decay is expected~\cite{BloNie89} to be exponentially fast for all $x>0$.

In the regime of power-law decay (sometimes called the critical regime), the scaling limit of the model should be described by (see e.g.~\cite[Section 5.6]{KagNie04}) a Conformal Loop Ensemble (CLE) of parameter $\kappa$ equal to
$$\kappa=\begin{cases} \frac{4\pi}{2\pi-\arccos(-n/2)}\in [\tfrac83, 4] &\text{ if }x=x_c(n),\\
\ \ \tfrac{4\pi}{\arccos(-n/2)}\in [4, 8] &\text{ if }x>x_c(n).\end{cases}$$
The regime $x=x_c(n)$ is sometimes referred to as the dilute critical regime (the limiting curves are simple) while the regime $x>x_c(n)$ is called the dense critical regime.

While the physical understanding of the loop $O(n)$ model is very advanced, the mathematical understanding  remains mostly limited to specific values of $n$:
\begin{itemize}[noitemsep]
\item For $n=1$, $x=1$, the model is equivalent to site percolation on the triangular lattice and it is proven~\cite{Smi01,CamNew06} that it converges to CLE(6) in the scaling limit.
\item For $n=1$, $0<x<1$, the model is in correspondence with the ferromagnetic Ising model on the triangular lattice. It is proven that for~$0<x<x_c(1)=1/\sqrt{3}$ the model is in the subcritical regime~\cite{AizBarFer87}, for~$x=1/\sqrt{3}$ it converges to CLE(3) in the scaling limit~\cite{Smi10,CheSmi12,CheDumHon14,BenHon16}, and for~$1/\sqrt{3}<x<1$ the model exhibits macroscopic loops (follows from the proof in~\cite{Tas14b}). Remarkably, the question of convergence to CLE(6) for~$1/\sqrt{3}<x<1$ remains open.
\item For $n=0$, the model is called the self-avoiding walk model (one has to make sense of the fact that the configuration does not contain any loops). It is known that the critical point is equal to $x_c(0)$ \cite{DumSmi12} and that the model is in a dense phase for $x>x_c(0)$ \cite{DumKozYad11}.
\item For large values of $n$ and suitable boundary conditions, it is proved \cite{DumPelSam14} that for any $x > 0$, the probability that the loop passing through a given vertex in $\Omega$ is of length $k$ decays exponentially fast in $k$ (though a phase transition of the hard-hexagon type~\cite[Chapter 14]{Bax89} takes place).

\item Finally, it is simple to show that there is exponential decay of loop lengths for all $n>0$ when $x$ is sufficiently small (see, e.g.,~\cite[Corollary 3.2]{DumPelSam14}).
 \end{itemize}
The goal of this paper is to study the loop $O(n)$ model in a wider regime of parameters. More precisely, we study the model for $n\ge1$ and $x\le \tfrac1{\sqrt n}$.

\subsection{Main results for the loop $O(n)$ model}

As mentioned above, the mathematical understanding of the model is quite limited, and until now, the loop~$O(n)$ model was not shown to exhibit macroscopic loops for~$n\in (1,2]$ at any $x>0$. The next theorem states that this holds at Nienhuis' critical point.  A measure $\bbP$ on loop configurations on $\bbH$ is called a Gibbs measure for the loop~$O(n)$ model with edge-weight~$x$, if for $\bbP$-almost any loop configuration $\xi$ and any finite subset $\Omega$ of edges of~$\bbH$,
\[
\bbP[\cdot \mid \calE(\Omega,\xi) ]=\bbP_{\Omega,n,x}^\xi.
\]
 For $k\in\bbN$, let $\Lambda_k$ be the ball in $\bbT$ of radius $k$ around the origin for the graph distance, and let $A_k$ be the annulus in $\bbH$ made of the edges of $\bbH$ between any two vertices belonging to some hexagon in $\Lambda_{2k} \setminus \Lambda_k$.

\begin{theorem}\label{thm:loop-macroscopic}
For $n\in[1,2]$ and $x=x_c(n)$, there exists $c>0$ such that for any $k>1$ and any loop configuration $\xi$,
\[
c\le \bbP_{A_k,n,x}^\xi[\exists\text{ a loop in $A_k$ surrounding 0}]\le 1-c.
\]
In particular, the Gibbs measure is unique and its samples almost surely have infinitely many loops going around the origin.
\end{theorem}
One can view Theorem~\ref{thm:loop-macroscopic} as evidence of a scale-invariant behaviour, supporting the conformal invariance conjecture of~\cite{KagNie04} stated above; at least for $n\in [1,2]$ and $x=x_c(n)$. In light of the conjecture, one expects that the conclusion of Theorem~\ref{thm:loop-macroscopic} remains in effect also for $n\in[0,2]$ and $x\ge x_c(n)$, while exponential decay of loop lengths takes place when $x<x_c(n)$. While it is expected that when increasing $x$ the model cannot transition from power-law decay of loop lengths to exponential decay, this seems difficult to prove (the measure $\bbP_{\Omega,n,x}^\xi$ is in general not monotonic in~$x$) and is currently only known for $n=1$ and $x\le 1$ (the ferromagnetic Ising model). Still, the theorem implies that (at least one) transition occurs for $n\in [1,2]$: Exponential decay takes place for small $x$ while power-law decay is present at $x=x_c(n)$.

\begin{figure}
	\centering
	\includegraphics[angle=90,height=11cm,width=12cm]{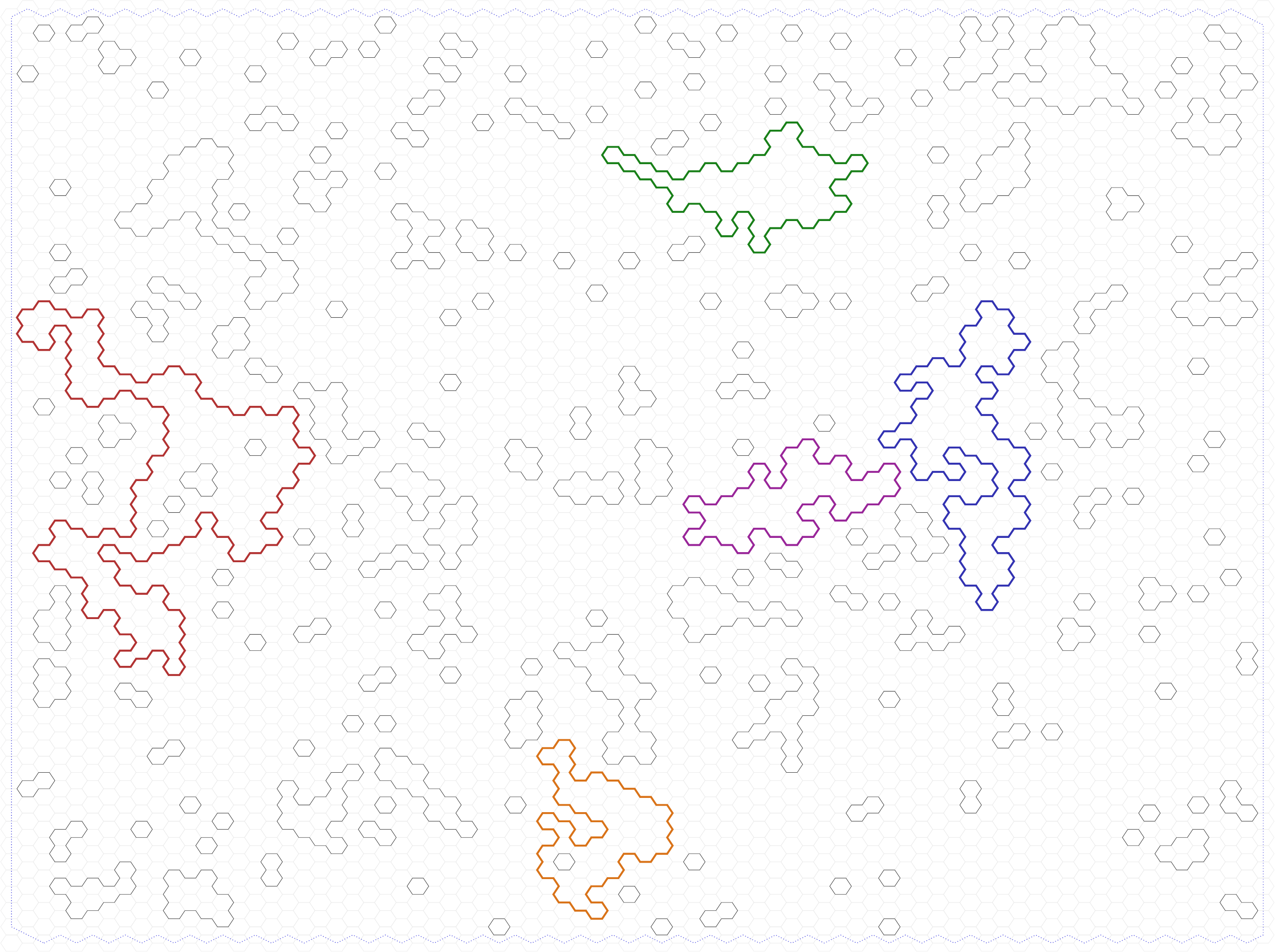}
	\caption{A sample of a random loop configuration on the critical line: $n=1.4$ and $x=x_c(n)\approx 0.6$. The longest loops are highlighted (from longest to shortest: red, blue, green, purple, orange). Theorem~\ref{thm:loop-macroscopic} shows that long loops are likely when $n \in [1,2]$ and $x=x_c(n)$.}
	\label{fig:loop-sample-critical}
\end{figure}
\begin{figure}
	\centering
	\begin{subfigure}[t]{.48\textwidth}
		\includegraphics[angle=90,height=7.5cm,width=\textwidth]{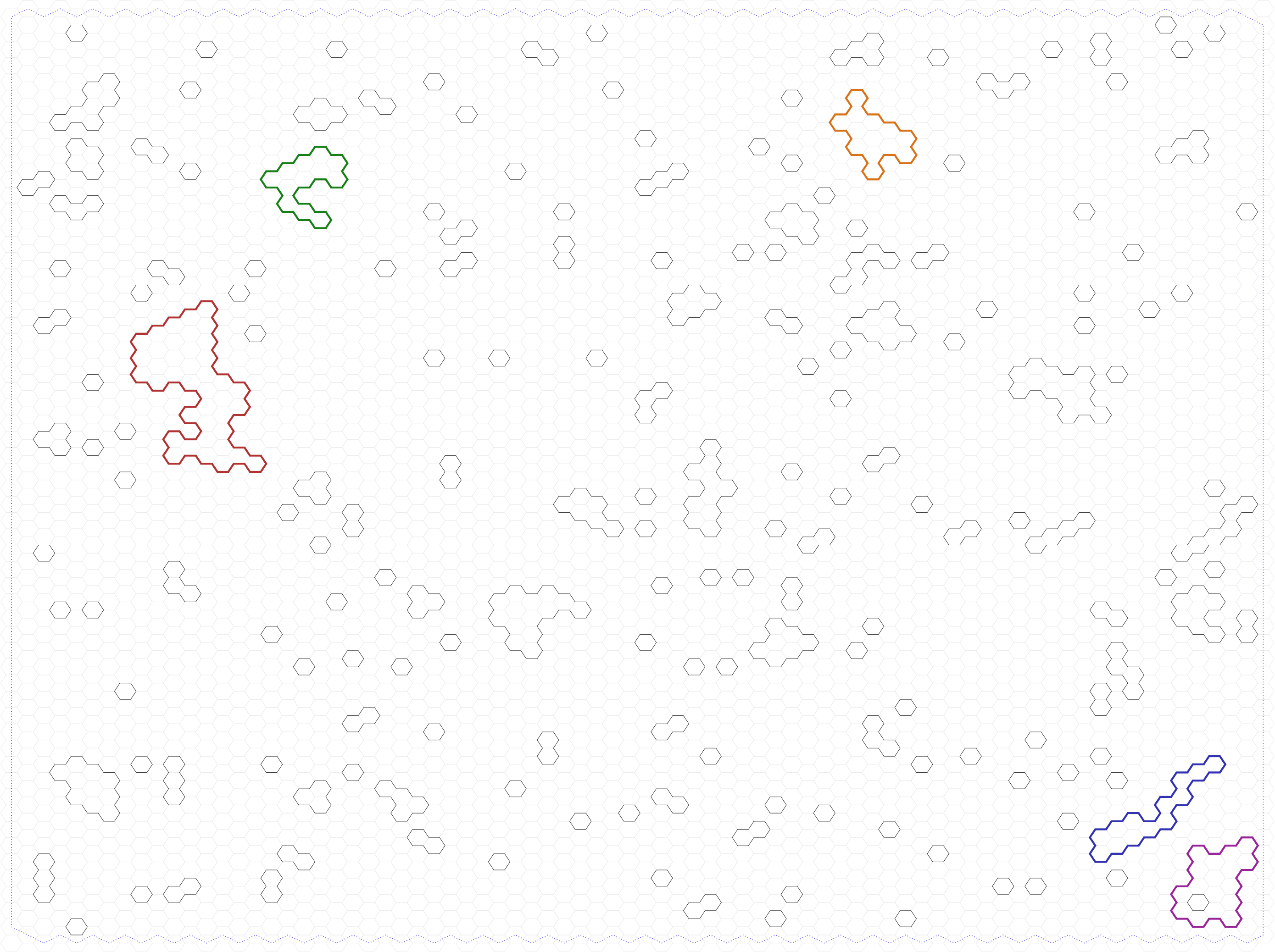}
		\caption{$n=1.4$ and $x=0.57<x_c(n)$.}
		\label{fig:loop-sample-n1.4-x0.57}
	\end{subfigure}%
	\begin{subfigure}{15pt}
		\quad
	\end{subfigure}%
	\begin{subfigure}[t]{.48\textwidth}
		\includegraphics[angle=90,height=7.5cm,width=\textwidth]{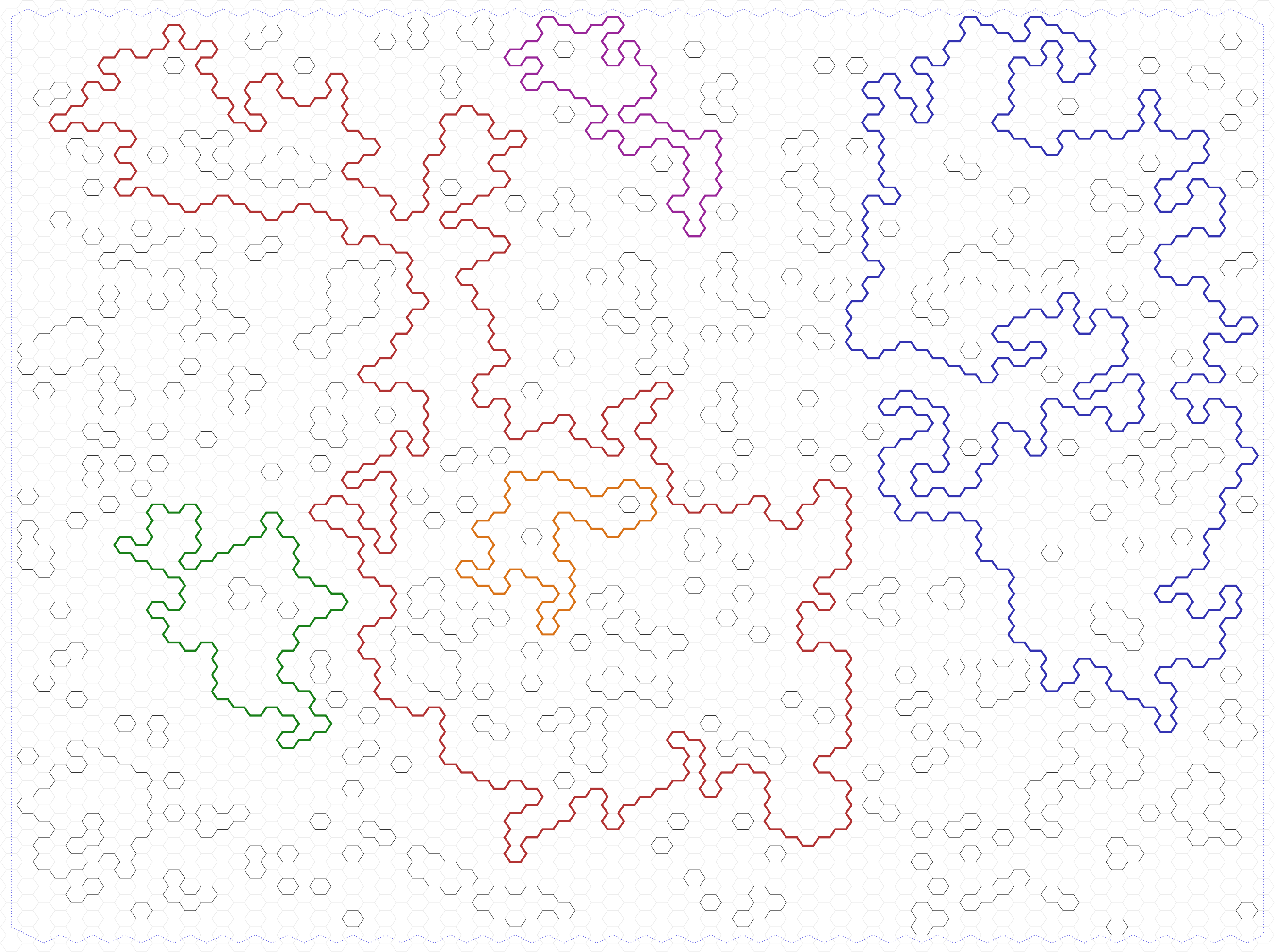}
		\caption{$n=1.4$ and $x=0.63>x_c(n)$.}
		\label{fig:loop-sample-n1.4-x0.63}
	\end{subfigure}
	\caption{Samples of random loop configuration below and above the critical line.}
	\label{fig:loop-samples}
\end{figure}

The proof of Theorem~\ref{thm:loop-macroscopic} combines probabilistic techniques with parafermionic observables. These observables first appeared in the context of the Ising model (where they are called order-disorder operators) and dimer models. They were later extended to the random-cluster model and the loop $O(n)$ model by Smirnov \cite{Smi06} (see \cite{DumSmi12a} for more details). They also appeared in a slightly different form in several physics papers going back to the early eighties \cite{FraKad80,BerLeC91} as well as in more recent papers studying a large class of models of two-dimensional statistical physics \cite{IkhCar09,RajCar07,RivCar06,Car09,IkhWesWhe13}.
They have been the focus of much attention in recent years and became a classical tool for the study of these models.

The precise property of these observables that will be used in this article is the fact that discrete contour integrals of parafermionic observables vanish for the special value of parameters $0\le n\le 2$ and $x=x_c(n)$.
Such an input was already used in~\cite{DumSmi12,Gla13} for the self-avoiding walk model, and in~\cite{DumSidTas13,Dum12} for random-cluster models. In our model, additional difficulties arise from the rigid structure of loop configurations. In order to overcome these difficulties, we develop a gluing technique, which, we hope, will be useful in the study of the loop~$O(n)$ model also when~$x\ne x_c(n)$.

The result for the loop $O(n)$ model on Nienhuis' critical line is derived from a clearer picture of the loop $O(n)$ model in the wider regime of parameters, $n\ge 1$, $x\le \tfrac1{\sqrt n}$.
This picture, in turn, is based on positive association (strong FKG) properties of the spin representation described in the next section.
These yield the following result, which includes the uniqueness of the translation-invariant (or even periodic) infinite-volume loop measure, as well as a dichotomy between two possible behaviors of the model --- exponential decay of loop lengths ({\bf A1}) vs. Russo--Seymour--Welsh type behavior ({\bf A2}). The two alternatives correspond to the predicted subcritical and critical (dilute or dense) behaviors of the model.

Let $\mathsf R$ be the largest diameter of a loop surrounding the origin (where~$\mathsf{R}=0$ if there is no such loop, and~$\mathsf{R}=\infty$ if there are infinitely many of them). A measure is \emph{periodic} if it is invariant under translations in a full-rank lattice.

\begin{theorem}\label{thm:loop-dichotomy}
For $n\ge1$, $x\le\tfrac1{\sqrt n}$, there exists a unique periodic Gibbs measure~$\bbP_{n,x}$ for the loop~$O(n)$ model with edge-weight~$x$. The measure~$\bbP_{n,x}$ is supported on loop configurations with no infinite paths, is extremal, is invariant to all automorphisms of $\bbH$, and can be obtained as a thermodynamical limit under empty boundary conditions. Furthermore, exactly one of the following occurs:
\begin{itemize}[noitemsep]
\item[\text{\bf A1}] There exists $c>0$ such that $\bbP_{n,x}[\mathsf R\ge k]\le\exp(-ck)$ for any $k\geq1$.
\item[\text{\bf A2}] There exists $c>0$ such that for any $k>1$ and any loop configuration $\xi$,
\begin{equation}\label{eq:RSW5}
c\le \bbP_{A_k,n,x}^\xi[\exists\text{ a loop in $A_k$ surrounding $0$}]\le 1-c.
\end{equation}
In particular, $\bbP_{n,x}$ is the unique Gibbs measure and $\mathsf R=\infty$ almost surely.
\end{itemize}
\end{theorem}
Both~\eqref{eq:RSW5} and {\bf P5} of Theorem~\ref{thm:cluster-dichotomy} below (from which~\eqref{eq:RSW5} is derived) should be understood as a box-crossing property; they imply many other properties of the model, including mixing at a power-law rate and fractal sub-sequential scaling limits. We refer to the corresponding results in \cite{DumSidTas13} for details. Also note that for $n\gg1$, the model was proved \cite{DumPelSam14} to satisfy~{\bf A1} for any $x\in(0,\infty)$.

When alternative~{\bf A2} holds, \eqref{eq:RSW5} implies the stronger statement that the weak limit of finite-volume measures under any boundary conditions is~$\bbP_{n,x}$. On the other hand, when alternative~{\bf A1} holds, we do not rule out the existence of non-periodic Gibbs measure.
We mention that in the case of~$n=2,x=1$, it is known that there is a unique Gibbs measure, but it remains open whether all weak limits coincide with it~\cite{GlaMan18}.

Alternative~{\bf A1} implies that the probability of having a loop surrounding the origin and entirely contained in a given domain is exponentially small for {\it some} boundary conditions. We expect this to hold for {\it any} boundary conditions and any (possibly non-simply connected) domain whenever~{\bf A1} is realized; see~\cite{GlaMan18b} for the proof for~$n\geq 1$ and~$x < \tfrac1{\sqrt{3}}+\varepsilon$.

\begin{remark}
   One may speculate that the length of loops in a domain is reduced, in a suitable sense, by adding a hole to the domain (with vacant boundary conditions along it). A natural attempt to prove such a statement then goes through the positive association of the spin representation described in the next section. This, however, does not seem to lead to the desired conclusion as the addition of the hole may be interpreted as restricting the spins on its boundary to take the same value, but such a restriction is not of ``definite sign'' and thus does not lead to a comparison with the initial distribution.
\end{remark}

We end this part of the introduction with a discussion of related models.

First, for certain values of $n$, the loop $O(n)$ model admits a nearest-neighbor representation. More precisely, when $n$ is the largest eigenvalue of the adjacency matrix of a graph, the loop $O(n)$ model is represented as the domain walls for a model on the triangular lattice with nearest-neighbor interactions (more precisely, face interactions). Taking the graph to be one of the ADE diagrams yields a representation with $n\in [1,2]$. Special cases include the dilute Potts model (of which more is said in the next section), the restricted Solid-On-Solid models and integer-valued Lipschitz height functions. ADE models were originally introduced in~\cite{Pas87}; see~\cite{Car10,Nie10} and~\cite[Section 3.3.2]{PelSpi17} for more information. Our results can then be recast in the language of these models.

We elaborate on the special case of Lipschitz height functions, arising when $n=2$. The functions are defined at the faces of~$\Lambda_k$, are normalized to~$0$ on the boundary of~$\Lambda_k$ and differ by~$1$, $0$ or~$-1$ at any two neighbouring faces. The probability of each function~$F$ is proportional to~$x$ to the number of pairs of adjacent faces~$u$ and~$v$ where~$F(u)\neq F(v)$. The loops represent the level lines of the height function, with each level line equally likely to be increasing or decreasing. Theorem~\ref{thm:loop-macroscopic} then implies that at~$x=1/\sqrt{2}$ there are typically $\log k$ level lines surrounding the origin and thus the height at the origin has fluctuations of order~$\sqrt{\log k}$. Recently, the same statement was proven in~\cite{GlaMan18} for~$x=1$ (uniform distribution over Lipschitz height functions); in contrast, the fluctuations were shown~\cite{GlaMan18b} to be bounded when~$x<1/\sqrt{3}+\varepsilon$ (corresponds to alternative~{\bf A1} in Theorem~\ref{thm:loop-dichotomy}).

Our results may further be compared with the phase diagram of the spin $O(2)$ model (the XY model). Following Berezinskii~\cite{Ber71}, Kosterlitz and Thouless~\cite{KosTho72,KosTho73}, and the celebrated rigorous proof by Fr\"ohlich and Spencer~\cite{FroSpe81}, the two-dimensional XY model is known to exhibit a phase transition from a regime with exponential decay of correlations at high temperature to a regime with power-law decay of correlations at low temperature --- the so-called Berezinskii-Kosterlitz-Thouless (BKT) transition. The loop $O(n)$ model is only an approximate graphical representation of the spin $O(n)$ model so results do not transfer between them. Still, the spin-spin correlation of the XY model is approximately, in the same sense as before, equal to the ratio between the partition function of the loop $O(2)$ model augmented by an additional path and the partition function of the usual loop $O(2)$ model; see~\cite[Eq. (2)]{DumPelSam14} for the precise formula. Similar ratios are considered in Section~\ref{sec:rel-weight} where they are shown to have a power-law lower bound. It is worth mentioning that obtaining such a lower bound is the main difficulty in the proof of the BKT transition for the XY model and that this is achieved, in~\cite{FroSpe81}, via the analysis of an integer-valued height function which is in an exact correspondence with the XY model. We mention that a different graphical representation is employed in~\cite{Cha98} to study ratios of partition functions of the XY model.

\subsection{The spin representation}
\label{sec:spin-rep}
As mentioned above, the loop $O(1)$ model can be seen as the Ising model on the triangular lattice~$\bbT$. More formally, the set of \emph{spin configurations} $\sigma=(\sigma_x:x\in\bbT)$ in $\{-1,1\}^\bbT$ is in bijection with the set $\calE(\bbH,\emptyset) \times \{-1,1\}$ of all loop configurations on~$\bbH$ via the mapping $\sigma \mapsto (\omega(\sigma),\sigma_0)$, where $\omega(\sigma)$ is the loop configuration composed of edges of $\bbH$ separating two hexagons $u$ and $v$ with $\sigma_u\ne\sigma_v$. In words, $\omega(\sigma)$ is the loop configuration obtained by taking the boundary walls between pluses and minuses. We use the denomination {\em plus} and {\em minus} for a vertex $x$ to denote the fact that the {\em spin} $\sigma_x$ is equal to $+1$ or $-1$, respectively.

In this section, we extend this correspondence to the loop $O(n)$ model for any $n>0$, by introducing a probability measure on spin configurations which is closely related to the loop $O(n)$ measure. We call this the \emph{spin representation} of the loop $O(n)$ model.

\begin{figure}
	\centering
	\begin{subfigure}[t]{.48\textwidth}
		\includegraphics[angle=90,height=7.5cm,width=\textwidth]{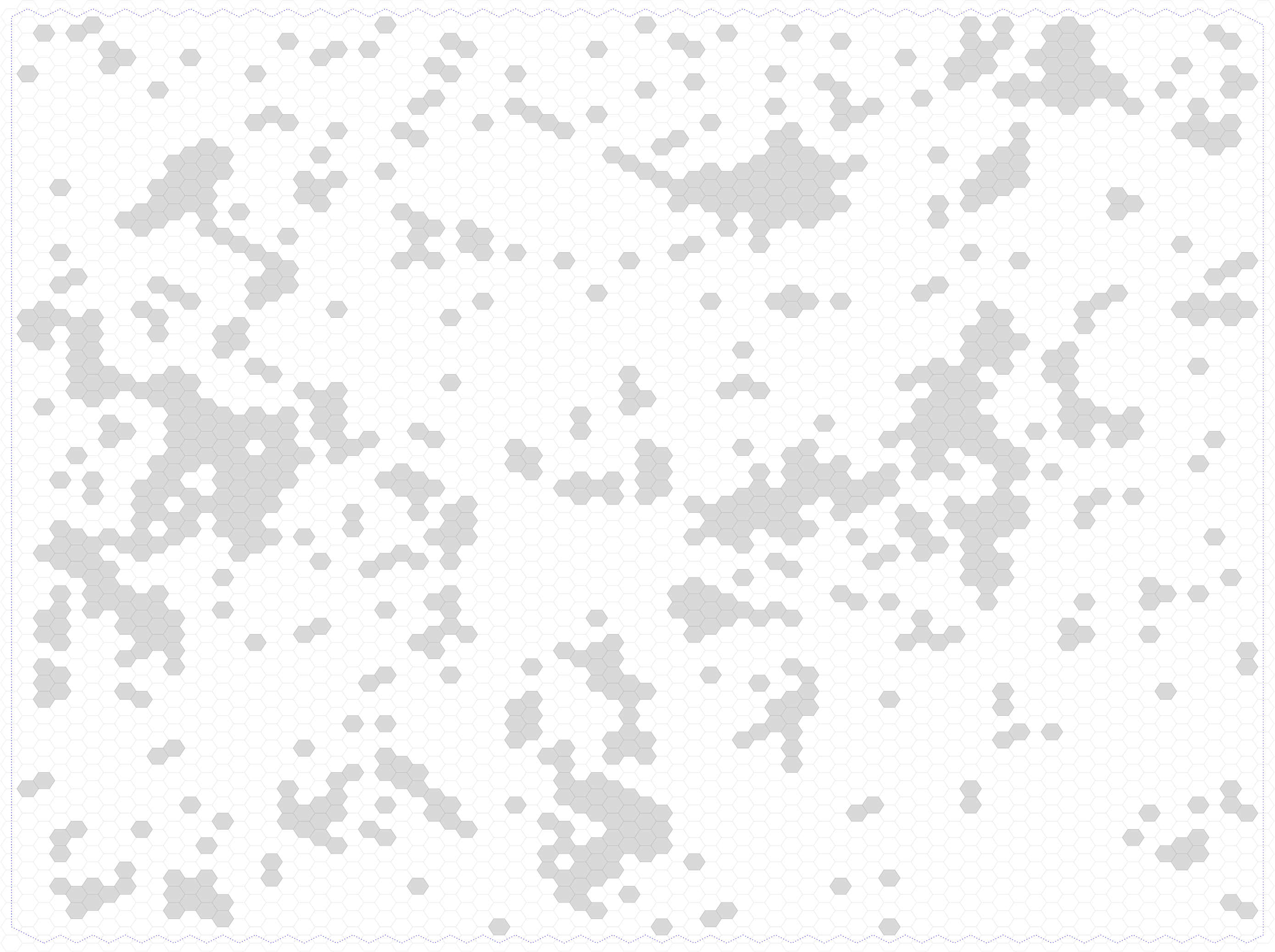}
		\caption{The spin representation.}
		\label{fig:loop-sample-cluster-repr-1}
	\end{subfigure}%
	\begin{subfigure}{15pt}
		\quad
	\end{subfigure}%
	\begin{subfigure}[t]{.48\textwidth}
		\includegraphics[angle=90,height=7.5cm,width=\textwidth]{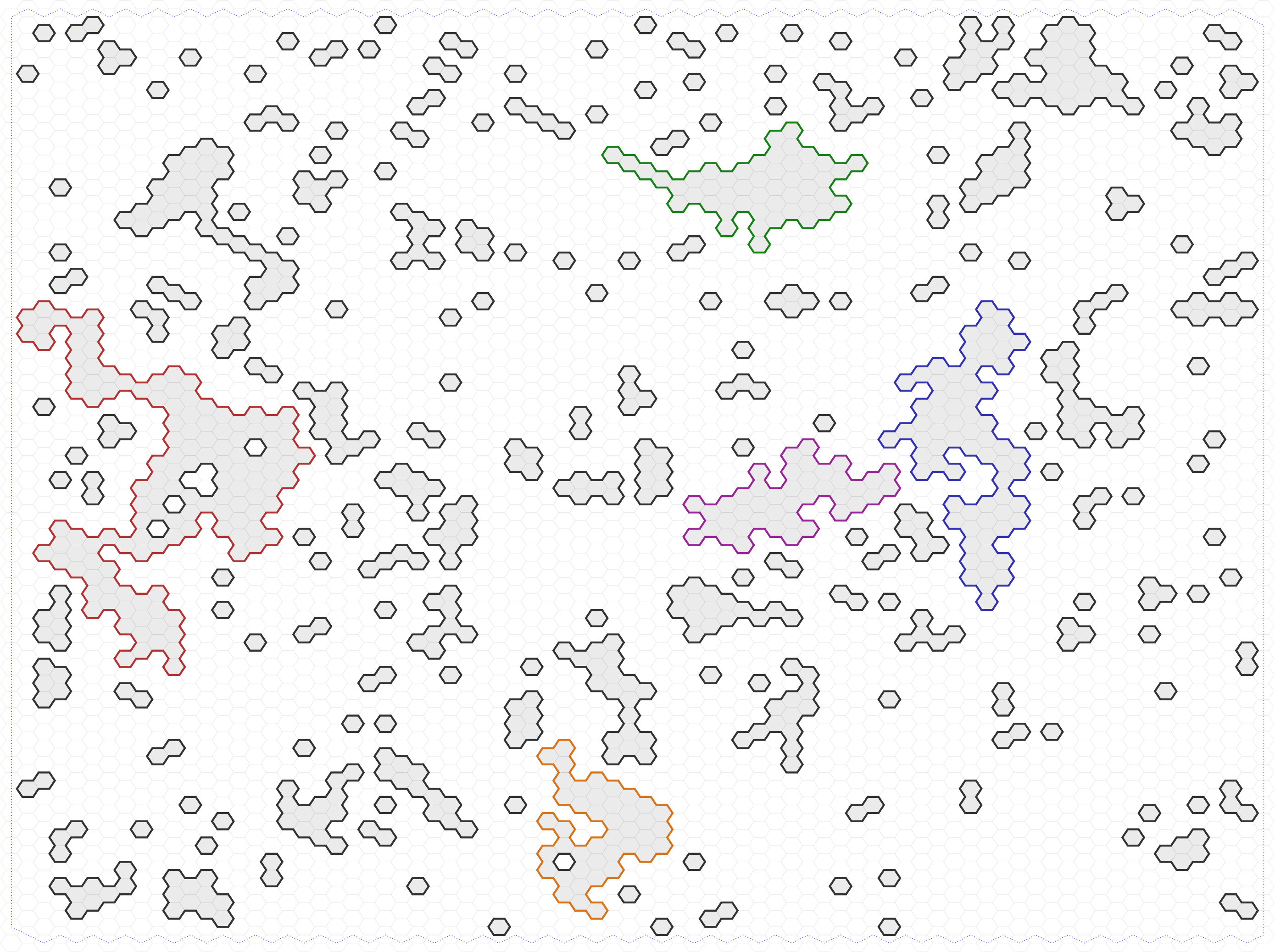}
		\caption{Loops and spins.}
		\label{fig:loop-sample-cluster-repr-2}
	\end{subfigure}
	\caption{Loop configurations on the hexagonal lattice are in bijection with colorings of the dual triangular lattice by two colors (up to a global permutation of the two colors): the loops are obtained from the coloring as the boundaries of clusters, and the coloring is obtained from the loops by switching color every time an edge of a loop is crossed.}

	\label{fig:loop-samples-cluster-repr}
\end{figure}

For~$\tau\in\{-1,1\}^\bbT$ and $G\subset \bbT$ finite, let $\Sigma(G,\tau)\subset \{-1,1\}^\bbT$ be the set of spin configurations that coincide with~$\tau$ outside of~$G$. The spin representation measure with edge-weight $x>0$ and loop-weight $n>0$ is the probability measure~$\mu_{G,n,x}^\tau$ on~$\Sigma(G,\tau)$ defined by the formula
\begin{equation}
\label{eq:dilute-Potts-measure}
\mu_{G,n,x}^\tau (\sigma) :=\frac{n^{k(\sigma)} x^{e(\sigma)}}{{\bf Z}_{G,n,x}^\tau}  ,\,
\end{equation}
for every $\sigma\in\Sigma(G,\tau)$, where~$k(\sigma)+1$ is the sum of the number of connected components of pluses and minuses in~$\sigma$ that intersect $G$ or its neighborhood, $e(\sigma) := \sum_{u \sim v} \1_{\sigma_u \neq \sigma_v}$ is the number of edges $\{u,v\}$ that intersect $G$ and have $\sigma_u \neq \sigma_v$, and ${\bf Z}_{G,n,x}^\tau$ is the unique constant making $\mu_{G,n,x}^\tau$ a probability measure. Clearly, both~$k(\sigma)$ and~$e(\sigma)$ depend on~$G$, but we omit it in the notation for brevity.

The next proposition states that~\eqref{eq:dilute-Potts-measure} indeed defines a representation of the loop $O(n)$ model.

\begin{proposition}
	\label{prop:bijection}
Let~$G \subset \bbT$ be finite and let $\Omega$ be the set of edges of $\bbH$ bordering a hexagon in~$G$. Then, for any $\tau\in \{-1,1\}^\bbT$ and any $n,x>0$, if $\sigma$ has law $\mu_{G,n,x}^\tau$, then $\omega(\sigma)$ has law $\bbP_{\Omega,n,x}^{\,\omega(\tau)}$.
\end{proposition}

\begin{proof}
The following combinatorial relations hold:
\begin{align*}
e(\sigma) = |\omega(\sigma)|
\quad\text{ and } \quad
k(\sigma) -  \ell(\omega(\sigma))=\#\{\text{infinite paths in $\omega(\sigma)$ intersecting $\Omega$}\},
\end{align*}
where the first equality is trivial and the second can be obtained by iteratively flipping signs in all finite clusters of~$\sigma$ which intersect~$G$ or are adjacent to $G$. Noting that the quantity on the right-hand side is constant for~$\sigma\in \Sigma(G,\tau)$ finishes the proof.
\end{proof}

An important property of the Ising model is its monotonicity (FKG inequality and monotonicity with respect to boundary conditions). This tool has become central in the study of the Ising model and luckily for us the spin representation shares this property with the Ising model for certain values of $x$ and $n$. Define a partial order on $\{-1,1\}^\bbT$ as follows: $\tau\le\tau'$ if $\tau_x\le \tau'_x$ for all $x\in \bbT$. We say that $A\subset \{-1,1\}^\bbT$ is increasing if its indicator function is an increasing function for this partial order.

\begin{theorem}\label{thm:FKG}
	Fix $n \ge 1$ and $nx^2 \le 1$. Then for any finite $G\subset\bbT$,
	\begin{itemize}[noitemsep]
		\item {\rm (strong FKG inequality)} for any $\tau \in \{-1,1\}^\bbT$ and any two increasing events $A$ and $B$,
		\begin{equation}
			\label{eq:FKG-thm}\tag{FKG}
			\mu_{G,n,x}^\tau(A\cap B)\ge\mu_{G,n,x}^\tau(A) \cdot \mu_{G,n,x}^\tau(B).
		\end{equation}
		\item {\em (comparison between boundary conditions)} for any $\tau\le\tau'$ and any increasing event $A$,
		$$\mu_{G,n,x}^\tau(A)\le \mu_{G,n,x}^{\tau'}(A).$$
	\end{itemize}
\end{theorem}
While fairly simple to prove, this theorem is our main toolbox for the study of the loop $O(n)$ model. In particular, it allows us to use techniques developed in~\cite{DumSidTas13} to prove the following dichotomy theorem for the spin representation. Before stating it, we remark that following this work, a similar FKG inequality was shown~\cite{GlaMan18} to hold when~$n\geq 2$, $(n-1)x^2 \leq 1$ and, through more intricate considerations, allowed to derive the dichotomy theorem for~$n=2$, $x=1$ (uniform Lipschitz functions).

By Theorem~\ref{thm:cluster-infinite-volume} below, infinite-volume limits~$\mu_{n,x}^+$ and~$\mu_{n,x}^-$ of~$\mu_{G,n,x}^+$ and~$\mu_{G,n,x}^-$ as $G\nearrow\bbT$ are well-defined, invariant under translations and ergodic.
Recall that $\Lambda_k \subset \bbT$ is the ball of radius $k$ around the origin. Write $V\longleftrightarrow W$ if some vertex of $V$ is connected to some vertex of $W$ by a path of adjacent pluses. We also write $v\longleftrightarrow\infty$ for the event that $v$ is in an infinite connected component of pluses.

\begin{theorem}\label{thm:cluster-dichotomy}
	For $n\ge1$ and $x\le\tfrac1{\sqrt n}$, the following conditions are equivalent:
	\begin{itemize}[noitemsep]
		\item[{\bf P1}] $\mu_{n,x}^+[0\longleftrightarrow\infty]=0$,
		\item[{\bf P2}] $\mu_{n,x}^-=\mu_{n,x}^+$,
		\item[{\bf P3}] $\sum_{v\in \bbT}\mu_{n,x}^-[0\longleftrightarrow v]=\infty$,
		\item[{\bf P4}] For any $v\in\bbT$, $$\lim_{k\rightarrow\infty}-\tfrac1k\log \mu_{n,x}^-[0\longleftrightarrow kv]=0,$$
		\item[{\bf P5}] There exists $c>0$ such that for any $k\ge 1$,
		$$\mu_{\Lambda_{2k},n,x}^-[\exists \text{ a circuit of neighboring pluses surrounding $\Lambda_k$ in $\Lambda_{2k}$}]\ge c.$$
	\end{itemize}
\end{theorem}
Similarly to the discussion of the box-crossing property in Theorem~\ref{thm:loop-dichotomy}, we wish to highlight the importance of Property {\bf P5}. It implies the decay of the probability of having an arm to distance $k$, as well as many other properties such as tightness of interfaces, universal exponents, etc. We again refer to \cite{DumSidTas13} for examples (and proofs) of applications in the context of the random-cluster model. Let us also remind the reader that {\bf P5} is equivalent to the following {\em box-crossing property} (which is itself related to the {\em Russo-Seymour-Welsh property}, see \cite{DumTas15c} for a review of recent advances on the subject): for  $\rho,\ep>0$,  there exists $c=c(\rho,\ep)>0$ such that for all $k\ge1$ and any $\tau\in\{-1,1\}^\bbT$,
\begin{equation}c\le \mu_{\overline R_k,n,x}^\tau[\exists \text{ a path of pluses crossing $R_k$ from left to right}]\le 1-c,\label{eq:crossing}\end{equation}
where $R_k$ and $\overline R_k$ are ``rectangles of $\bbT$'' defined by
\begin{align*}R_k&:=\{r+{\rm e}^{{\rm i}\pi/3}s~:~0\le r\le k~,~0\le s\le \rho k\},\\
	\overline R_k&:=\{r+{\rm e}^{{\rm i}\pi/3}s~:~-\ep k\le r\le (1+\ep)k~,~-\ep k\le s\le (\rho+\ep) k\}.
\end{align*}
We also remark that Theorem~\ref{thm:loop-macroscopic} shows that, when $n \in [1,2]$ and $x=x_c(n)$, condition {\bf P5} holds, and hence also conditions {\bf P1}-{\bf P4}.

\medskip
To better understand the critical nature of the loop $O(n)$ model it is useful to view it as a particular case of a wider family of models, which is obtained when certain parameters are tuned to specific values (in the spirit of adding a magnetic field to a spin system, or viewing the \emph{critical} random-cluster model as a line in the general $q,p$ parameter space). To this end, it is natural to introduce two external fields $h, h'$, as follows.
The spin representation measure with edge-weight $x>0$, loop-weight $n>0$ and external fields $h,h' \in \R$ is the probability measure~$\mu_{G,n,x,h,h'}^\tau$ on~$\Sigma(G,\tau)$ defined by the formula
\begin{equation}
\label{eq:cluster-measure-with-magnetization}
\mu_{G,n,x,h,h'}^\tau (\sigma) :=\frac{n^{k(\sigma)} x^{e(\sigma)} e^{h r(\sigma) + h'r'(\sigma)}}{{\bf Z}_{G,n,x,h,h'}^\tau}  ,\,
\end{equation}
where~$r(\sigma) := \sum_{u \in G} \sigma_u$ is the sum of spins of $\sigma$ in $G$, $r'(\sigma) := \tfrac12 \sum_{ t=\{u,v,w\} } \sigma_u \1_{\sigma_u=\sigma_v=\sigma_w}$ is one-half of the difference between the number of plus and minus monochromatic triangles that intersect $G$ (where a monochromatic triangle is a set of three mutually adjacent vertices with equal spins), and ${\bf Z}_{G,n,x,h,h'}^\tau$ is the unique constant making $\mu_{G,n,x,h,h'}^\tau$ a probability measure.

We detail two motivations for the above model. First,  in~\cite{Nie91}, Nienhuis discusses the dilute Potts model. Its vacancy/occupancy representation is in a direct correspondence with the model~\eqref{eq:cluster-measure-with-magnetization}, allowing our results to be interpreted in that context. The loop~$O(n)$ model can be viewed as the self-dual surface of the vacancy/occupancy representation as the distribution at $h=h'=0$ is invariant under a global spin flip. Nienhuis predicts that this surface is also critical and that the line~$x=x_c(n)$ is \emph{tricritical} in the sense that the order of the phase transition changes there. Our theorems partially confirm these predictions.

A second motivation comes from the Hammersley-Clifford theorem~\cite{hammersley1971markov}, which shows that if a Markov random field on the triangular lattice has positive density then this density factorizes as a product over triangle interaction terms. This implies that, in the case $n=1$, the representation~\eqref{eq:cluster-measure-with-magnetization} is the most general form of a homogeneous $\{-1,1\}$-valued Markov random field with a positive probability density.

In Proposition~\ref{prop:FKG}, we show that the strong FKG inequality extends to the case of the spin representation measure with an external field if~$nx^2\leq e^{-|h'|}$.  This enables us once again to use the techniques developed for the random-cluster model and to define Gibbs measures~$\mu_{n,x,h,h'}^+$ and~$\mu_{n,x,h,h'}^-$ for the spin representation as weak limits as~$G\nearrow\bbT$ of finite-volume measures $\mu^+_{G,n,x,h,h'}$ and~$\mu^-_{G,n,x,h,h'}$, corresponding to $\tau\equiv +1$ and to $\tau\equiv -1$, respectively.

\begin{theorem}\label{thm:cluster-infinite-volume}
	For any $(n,x,h,h')$ such that $n \ge 1$ and $nx^2 \le e^{-|h'|}$, there exists a Gibbs measure $\mu_{n,x,h,h'}^+$ for the spin representation satisfying the following properties:
	\begin{itemize}[noitemsep]
		\item $\mu_{n,x,h,h'}^+$ is the weak limit of the measure $\mu^+_{G,n,x,h,h'}$ as $G\nearrow\bbT$.
		\item $\mu_{n,x,h,h'}^+$ is extremal and invariant under all automorphisms of $\bbT$.
		\item the $\mu_{n,x,h,h'}^+$-probability that there exists both an infinite connected component of pluses and an infinite connected component of minuses is 0.
	\end{itemize}
	Similarly, there exists a measure $\mu_{n,x,h,h'}^-$ (possibly equal to $\mu_{n,x,h,h'}^+$) satisfying the analogous properties.	
	
	Moreover, any periodic Gibbs measure for the spin representation is a mixture of the two measures~$\mu_{n,x,h,h'}^+$ and~$\mu_{n,x,h,h'}^-$.
	We use the notation~$\mu_{n,x}^+:=\mu_{n,x,0,0}^+$ and~$\mu_{n,x}^-:=\mu_{n,x,0,0}^-$.
\end{theorem}

We remark that, since $r(\sigma)$ and $r'(\sigma)$ are anti-symmetric, the map $\sigma \mapsto -\sigma$ takes the measure $\mu_{G,n,x,h,h'}^\tau$ to $\mu_{G,n,x,-h,-h'}^{-\tau}$.
In particular, $h=h'=0$ is a  self-dual surface in the space of parameters.
Recall that $h$ can be interpreted as an external field favoring pluses over minuses.
Comparing the spin representation defined above to the well-known random-cluster model (also known as the FK-model), $h$ plays an analogous role as the parameter $p$ of the random-cluster model (more precisely, $e^h$ should be compared to $\tfrac{p}{1-p}$). Similarly, $+$ and $-$ boundary conditions correspond respectively to the wired and free boundary conditions of the random-cluster model.
For certain properties, the analogy is rather direct: one may use the suitably modified techniques of the random-cluster model --- the key point is to obtain the monotonicity properties of the spin representation (the FKG inequality and the comparison between boundary conditions stated above). However, in order to show for~$n\in[1,2]$ and~$x=x_c(n)$ the existence of macroscopic clusters of pluses in case of minus boundary conditions (\textbf{P5} of Theorem~\ref{thm:cluster-dichotomy}), we found it necessary to consider the specific properties of the loop~$O(n)$ model and develop the gluing technique (see Section~\ref{sec:theorem-big-loops}).

The next theorem shows that, within the $h'=0$ surface, the self-dual line $h=0$ is critical.

\begin{theorem}\label{thm:cluster-h-critical}
	For $n\ge1$ and $x\le\tfrac1{\sqrt n}$,
	\begin{itemize}[noitemsep]
		\item if $h>0$, $\mu_{n,x,h,0}^-[0\longleftrightarrow\infty]>0$.
		\item if $h<0$, there exists $c_h>0$ such that for all $v\in\bbT$,
		$$\mu_{n,x,h,0}^+[0\longleftrightarrow v]\le \exp[-c_h\,d(v,0)].$$
	\end{itemize}
\end{theorem}
This result is similar to the recent developments in the understanding of random-cluster models, for which the critical point was computed on the square lattice; see \cite{BefDum12,DumRaoTas17}.

\paragraph{Organization.} The paper is organized as follows. The next two sections describe the proofs of Theorems~\ref{thm:loop-dichotomy}--\ref{thm:cluster-h-critical}. The last section introduces parafermionic observables and presents the proof of Theorem~\ref{thm:loop-macroscopic}.

\subsection*{Acknowledgments.}
We are grateful to Ioan Manolescu for pointing out a mistake in the proof of an earlier version of Theorem~\ref{thm:cluster-infinite-volume}, which stated a characterization of all (as opposed to only periodic) Gibbs measures.

Research of H. D.-C. was funded by a IDEX Chair from Paris Saclay and by the NCCR SwissMap from the Swiss NSF. Research of A. G. was supported by the Swiss NSF grant P2GE2\_165093, and  partially supported by the European Research Council starting grant 678520 (LocalOrder); part of the work was conducted during the visits of A. G. to the University of Geneva and he is grateful for its hospitality. Research of R.P. was supported by Israeli Science Foundation grant 861/15 and the European Research Council starting grant 678520 (LocalOrder). Research of Y.S. was supported by Israeli Science Foundation grant 861/15, the European Research Council starting grant 678520 (LocalOrder), and the Adams Fellowship Program of the Israel Academy of Sciences and Humanities.

\section{FKG inequality and comparison between boundary conditions}

This section is devoted to monotonicity properties of the spin representation. Theorem~\ref{thm:FKG} follows directly from Proposition~\ref{prop:FKG} and Corollary~\ref{cor:CBC} below. We start by proving the {\em Fortuin-Kasteleyn-Ginibre lattice condition} which is known to imply~\eqref{eq:FKG-thm} by~\cite[Theorem (2.19)]{Gri06}. For~$\sigma,\sigma'\in \{-1,1\}^\bbT$, we define~$\sigma\vee \sigma'$ and $ \sigma\wedge \sigma'$ by
\begin{equation}
(\sigma\vee \sigma')(v) := \max \{\sigma(v), \sigma(v')\},
\quad
(\sigma\wedge \sigma') (v) := \min \{\sigma(v), \sigma(v')\},
\quad v\in\bbT.
\end{equation}

\begin{proposition}[FKG lattice condition]\label{prop:FKG}
Fix $(n,x,h,h')$ such that $n \ge 1$ and $nx^2 \le e^{-|h'|}$. Let~$B\subset\bbT$ be such that each two neighboring vertices in~$B$ have a common neighbor inside~$B$. Let $G\subset B$ be finite, and $\tau\in\{-1,1\}^B$.
Then, for every~$\sigma,\sigma' \in\{-1,1\}^B$ such that~$\sigma_{|B\setminus G} = \sigma'_{|B\setminus G}$,
\begin{equation}\label{eq:FKG}
\mu_{G,n,x,h,h'}^\tau[\sigma\vee \sigma'] \cdot \mu_{G,n,x,h,h'}^\tau[\sigma\wedge \sigma']\ge \mu_{G,n,x,h,h'}^\tau[\sigma] \cdot \mu_{G,n,x,h,h'}^\tau[\sigma'].
\end{equation}
\end{proposition}

\begin{remark}
The previous proposition states the strong FKG inequality for the spin representation defined by~\eqref{eq:cluster-measure-with-magnetization} in the case $B=\bbT$. When extending the inequality to the case~$B\subset \bbT$, we slightly abuse notation by using~$\mu_{G,n,x,h,h'}^\tau(\sigma)$ for~$\sigma, \tau$ defined only on a subset~$B$ of~$\bbT$ containing~$G$. By this we mean that~$\mu_{G,n,x,h,h'}^\tau(\sigma)$ is defined by~\eqref{eq:cluster-measure-with-magnetization}, where~$k(\sigma)$, $e(\sigma)$, $r(\sigma)$ and~$r'(\sigma)$ are defined in the same way. This extension will be instrumental in Corollary~\ref{cor:CBC}, where we prove monotonicity in boundary conditions.
\end{remark}

\begin{proof}
By~\cite[Theorem (2.22)]{Gri06}, it is enough to show the inequality for any two configurations which differ in exactly two places i.e.,~that for any $\sigma\in \Sigma(G,\tau)$ and $u\ne v$ in $G$,
\begin{align}
\label{eq_FKG}
\mu_{G,n,x,h,h'}^\tau[\sigma^{++}] \cdot \mu_{G,n,x,h,h'}^\tau[\sigma^{--}] \ge \mu_{G,n,x,h,h'}^\tau[\sigma^{+-}] \cdot \mu_{G,n,x,h,h'}^\tau[\sigma^{-+}] ,
\end{align}
where $\sigma^{\eta\eta'}$ is the configuration coinciding with $\sigma$ except (possibly) at $u$ and $v$, and such that $\sigma_u^{\eta\eta'}= \eta$ and $\sigma_v^{\eta\eta'}= \eta'$.
Equivalently, one needs to prove that
\begin{equation}
\label{eq_FKG2}
(\log n)\Delta k +  (\log x)  \Delta e+ h \Delta r + h' \Delta r' \ge 0 ,
\end{equation}
where
\[ \Delta k := k(\sigma^{++}) + k(\sigma^{--}) - k(\sigma^{+-}) - k(\sigma^{-+}) ,\]
and $\Delta e$, $\Delta r$ and $\Delta r'$ are defined similarly.
Observe that $\Delta r=0$ so that we may drop this term in \eqref{eq_FKG2}.

Write $\Delta k = \Delta k^+ + \Delta k^-$, where $\Delta k^+$ and $\Delta k^-$ take into account the plus or minus connected components separately.
Clearly, only plus-clusters containing~$u$ or~$v$ or adjacent to one of these vertices contribute to~$\Delta k^+$. It is easy to see that each such cluster in~$\sigma^{+-}$ or~$\sigma^{-+}$ is also a cluster in~$\sigma^{--}$ as soon as it does not intersect~$\{u,v\}$. The number of plus-clusters intersecting~$\{u,v\}$ is equal to one in~$\sigma^{+-}$ and~$\sigma^{-+}$ and is at least one in~$\sigma^{++}$, whence~$\Delta k^+ \ge -1$. Moreover, $\Delta k^+ = -1$ only if there are no plus-clusters in~$\sigma^{--}$ that are adjacent to both~$u$ and~$v$, and if~$u$ and~$v$ are in the same plus-cluster of~$\sigma^{++}$. In other words, $\Delta k^+<0$ implies that $\Delta k^+ = -1$, $u$ and $v$ are adjacent, and common neighbors of $u$ and $v$ have spin $-1$. The analogous statement holds for $\Delta k^-$.

We now divide the study into three cases.
\begin{itemize}
\item Assume $u$ and $v$ are not neighbors. Then, $\Delta e=\Delta r'=0$ and $\Delta k^+,\Delta k^-\ge0$. The assumption that $n\ge1$ immediately implies \eqref{eq_FKG2}.

\item Assume $u$ and $v$ are neighbors and have two common neighbors with different spins.
Then, $\Delta r'=0$, $\Delta e=-2$ and $\Delta k \ge 0$. Since $n \ge 1$ and $nx^2 \le 1$, we get~\eqref{eq_FKG2}.

\item Assume $u$ and $v$ are neighbors and common neighbors of $u$ and $v$ have the same spin.
Then, $|\Delta r'|\le 1$, $\Delta e=-2$ and $\Delta k\ge -1$ (since either $\Delta k^+$ or $\Delta k^-$ is non-negative). Since $n \ge 1$ and $nx^2 \le e^{-|h'|}$, we get~\eqref{eq_FKG2}. \qedhere
\end{itemize}
\end{proof}

\begin{remark}
It is easy to see that the conditions $n \ge 1$ and $nx^2 \le e^{-|h'|}$ are necessary in order for the FKG lattice condition to hold for arbitrary $G\subset\bbT$.
\end{remark}

The following corollary will be important in the proof of Lemma~\ref{lem_cutting_path_at_boundary}. It compares the probabilities of the events that the spins of two sets $U$ and $V$ are equal to a certain value.

\begin{corollary}\label{cor:FKG-several-faces}
Fix $(n,x,h,h')$ such that $n \ge 1$ and $nx^2 \le e^{-|h'|}$.
Let $G\subset\bbT$ be finite and $\tau\in\{-1,1\}^\bbT$.
Then, for every~$\sigma,\sigma'\in \Sigma(G,\tau)$ and~$U,V\subset G$,
\begin{align}
\label{eq-FKG-several-faces}
\mu_{G,n,x,h,h'}^\tau[\sigma_{|U} =\sigma_{|V} = 1 ] \,\cdot \, &\mu_{G,n,x,h,h'}^\tau[\sigma_{|U} =\sigma_{|V} = -1 ] \\
&\ge \mu_{G,n,x,h,h'}^\tau[\sigma_{|U} = 1, \sigma_{|V} = -1 ]\cdot\mu_{G,n,x,h,h'}^\tau[\sigma_{|U} = -1, \sigma_{|V} = 1 ].
\end{align}
\end{corollary}

\begin{proof}
Trivially,~\eqref{eq:FKG} implies that the FKG lattice condition is satisfied also for the conditioned measure~$\nu:=\mu_{G,n,x,h,h'}^\tau[\, \cdot \mid \sigma_{|U}\equiv \text{const}, \sigma_{|V}\equiv \text{const}]$, and hence this measure satisfies the FKG inequality (see~\cite[Theorem (2.19)]{Gri06}), i.e., for any two increasing events~$A,B\subset \{-1,1\}^\bbT$,
\[
\nu[A\cap B] \ge \nu[A] \cdot \nu[B].
\]
Applying this inequality to~$A:=\{\sigma_{|U}=1\}$ and~$B:=\{\sigma_{|V}=1\}$, yields the inequality
\begin{align*}
\nu[\sigma_{|U} = \sigma_{|V} = 1] \ge \nu[\sigma_{|U} = 1] \cdot \nu[\sigma_{|V} = 1],
\end{align*}
which can be written in the form~\eqref{eq-FKG-several-faces}, where~$\mu_{G,n,x,h,h'}^\tau$ is replaced with~$\nu$. Removing the redundant condition finishes the proof.
\end{proof}

In order to treat boundary conditions, we recall the following {\em domain Markov property} (the proof is straightforward and therefore omitted).
For any $(n,x,h,h')$, any finite $H\subset G\subset \bbT$ and any~$\tau,\sigma\in\{-1,1\}^\bbT$,
$$
\mu_{G,n,x,h,h'}^\tau[\sigma \mid \sigma_{|\bbT\setminus H}=\tau_{|\bbT\setminus H}]=\mu_{H,n,x,h,h'}^{\tau}[\sigma].$$
\begin{remark}As a consequence of this property and the definition of the measure, the model satisfies the {\em finite energy property}:
for any $\tau\in\{-1,1\}^\bbT$ and $\sigma \in \Sigma(G,\tau)$, $\mu^\tau_{G,n,x,h,h'}[\sigma]\ge \ep^{|G|}$ for a constant $\ep>0$ depending only on $(n,x,h,h')$.
\end{remark}
Let us conclude this section by observing that the domain Markov property together with the FKG lattice condition imply the following {\em comparison between boundary conditions}.

\begin{corollary}[Comparison between boundary conditions]\label{cor:CBC}
Consider $G\subset\bbT$ finite and fix $(n,x,h,h')$ such that~$n\ge 1$ and~$nx^2\leq e^{-|h'|}$. For any increasing event $A$ and any $\tau\le \tau'$,
 $$\mu_{G,n,x,h,h'}^\tau[A]\le \mu_{G,n,x,h,h'}^{\tau'}[A].$$
\end{corollary}

\begin{proof}
There exists~$B\subset\bbT$ finite such that~$G\subset B$ and for any~$\sigma\in\Sigma(G,\tau)\cup\Sigma(G,\tau')$, the number $k(\sigma)$ is not changed by removing all hexagons outside~$B$. It is enough to prove the inequality for measures~$\mu_{G,n,x,h,h'}^\tau$ and~$\mu_{G,n,x,h,h'}^{\tau'}$ on configurations restricted to~$B$. As in Proposition~\ref{prop:FKG}, we abuse notation and keep denoting measures in the same way. Consider the finite set $H:=\{x\in B\setminus G:\tau_x<\tau'_x\}$. The domain Markov property implies that
\begin{align*}\mu_{G,n,x,h,h'}^\tau&=\mu_{G\cup H,n,x,h,h'}^\tau[\, \cdot \mid \sigma_{|H}=-1],\\
\mu_{G,n,x,h,h'}^{\tau'}&=\mu_{G\cup H,n,x,h,h'}^\tau[\, \cdot \mid \sigma_{|H}=1].\end{align*} As a consequence, the FKG inequality~\eqref{eq:FKG} applied to configurations restricted to the set~$B$ implies that
\begin{equation}
\mu_{G,n,x,h,h'}^\tau[A]\le \mu_{G\cup H,n,x,h,h'}^\tau[A]\le \mu_{G,n,x,h,h'}^{\tau'}[A].\tag*{\qedhere}
\end{equation}
\end{proof}

\section{Proofs of Theorems~\ref{thm:loop-dichotomy} and \ref{thm:cluster-dichotomy}--\ref{thm:cluster-h-critical}}
Now that we are in possession of the FKG inequality and the comparison between boundary conditions, the proofs of Theorems~\ref{thm:cluster-dichotomy}--\ref{thm:cluster-h-critical} follow standard paths already described in detail in the literature. For this reason, we only outline the arguments and give the relevant references.

\begin{proof}[Proof of Theorem~\ref{thm:cluster-infinite-volume}] We fix~$n,x,h,h'$ and omit them everywhere in the notation. The first two items are very simple consequences of the comparison between boundary conditions (Corollary~\ref{cor:CBC}) and the domain Markov property. In particular, proofs that are valid for the random-cluster model also apply here. We refer to Theorem~(4.19) and Corollary~(4.23) in~\cite{Gri06}. The extremality of $\mu^+$ and $\mu^-$ implies that these measures inherit the positive association property of their finite-volume counterparts $\mu^-_G$ and $\mu^+_G$.

Let us now turn to the third item. First, the measure is ergodic and satisfies the finite energy property. As a consequence, the Burton-Keane argument \cite{BurKea89} shows that the infinite connected component of pluses, when it exists, is unique (see \cite[Theorem (5.99)]{Gri06} for an exposition of the argument). Similarly, the infinite connected component of minuses, when it exists, is unique. Thus, there cannot be coexistence of an infinite connected component of pluses and an infinite connected component of minuses, since Zhang's construction \cite[Theorem (6.17)]{Gri06} would imply the existence of more than one infinite connected component of pluses.

As for the random-cluster model~\cite[Theorem (4.31)]{Gri06}, any weak limit of finite-volume measures which has at most one infinite cluster (of each sign) is a Gibbs measure. Thus, by what we have shown above, $\mu^+$ is a Gibbs measure.

Corollary~\ref{cor:CBC} implies that for any finite $G$ and $\tau\in\{-1,1\}^\bbT$, the measure $\mu^\tau_G$ is stochastically between $\mu^-_G$ and $\mu^+_G$. Thus if $\mu^- = \mu^+$ then the model has a unique infinite-volume limit and, in particular, a unique Gibbs measure.

It remains to consider the case that $\mu^- \neq \mu^+$ and prove that any periodic Gibbs measure is a mixture of these two measures. For the two-dimensional Ising model, the stronger statement that any (possibly non-periodic) Gibbs measure is a mixture of the plus and minus measures, was proven by Aizenman~\cite{Aiz80} and Higuchi~\cite{Hig81}. Both these proofs rely on particular properties of the Ising model and do not apply to our case. Instead, we adapt the later proof by Georgii--Higuchi~\cite{GeoHig00}, which is more geometric and can be extended to the context of dependent models on the triangular lattice. Specifically, we adapt the proofs of Lemma 2.1, Lemma 2.2, Lemma 3.1 and Corollary 3.2 of~\cite{GeoHig00} to our situation. Below, we use the notation of~\cite{GeoHig00}, replacing *-connectivity in $\bbZ^2$ with standard connectivity in $\bbT$.

The main difference between the spin representation and the Ising model is that the latter has the domain Markov property, which states that the distribution in finite volume with prescribed boundary values is completely determined by one layer of spins on the boundary of the volume.
The formula~\eqref{eq:cluster-measure-with-magnetization} shows that this is not the case for the spin representation, as the quantity $k(\sigma)$ (defined after~\eqref{eq:dilute-Potts-measure}) which appears there may depend on the boundary values beyond the first layer. Nevertheless, a partial Markov property is available for the spin representation which suffices in order to adapt the proofs of~\cite{GeoHig00}. For a finite $G\subset\T$ and $\tau\in\{-1,1\}^\T$, the finite-volume measure $\mu^\tau_G$ depends on $\tau$ only through its first layer of spins outside $G$ when that first layer can be partitioned into two \emph{connected} sets such that $\tau$ is constant on each of these sets. Indeed, it is straightforward that in this case $k(\sigma)$ does not depend on the spins in $\tau$ beyond the first layer.

Lemma 2.1 of~\cite{GeoHig00} states that any Gibbs measure $\mu\neq\mu^{-}$ gives positive probability to the event that an infinite $+1$ cluster exists. Its statement and proof apply to our situation verbatim, using the partial Markov property above. In particular, as we assumed that $\mu^-\neq\mu^+$, the lemma implies that samples from $\mu^+$ have an infinite $+1$ cluster and samples from $\mu^-$ have an infinite $-1$ cluster, almost surely. Consequently, as there is no coexistence of infinite clusters of both signs, the measures $\mu^+$ and $\mu^-$ are not invariant under the $T$ transformation (flipping of all signs). This is used in the proof of Lemma 3.1 of~\cite{GeoHig00}.

The statement of Lemma~2.2 needs to be modified as follows: Let $\omega$ be sampled from a Gibbs measure $\mu$, let $\pi$ be a half-plane in $\T$ and let $R$ be the reflection through the boundary of $\pi$ (so that~$\pi$ and $R(\pi)$ cover the entire plane and have a line of $\T$ in common). Suppose that for every finite $\Delta\subset\pi$ there is a finite, connected, $R$-invariant $G$ with $\Delta\subset G$ such that $\omega\equiv 1$ on the part of the external vertex boundary of $G$ which is in $\pi$. Then $\mu$ stochastically dominates $\mu\circ R\circ T$.

The proof is a modification of the argument in~\cite{GeoHig00}: First find, in a large $R$-invariant $\Lambda$, the maximal connected, $R$-invariant $G\subset\Lambda$ for which the assumption holds. Such a $G$ exists with probability close to $1$ when $\Lambda$ is large and we proceed on the event that it exists. Condition on $\omega$ outside of $G$ and note that the distribution of $\omega|_G$ equals $\mu^\omega_G$ by the maximality of $G$ (as its boundary can be explored from the outside). Thus, Corollary~\ref{cor:CBC} implies that the distribution of $\omega|_G$ stochastically dominates $\mu^{\omega'}_G$, where $\omega'$ coincides with $\omega$ on $\pi$ and equals $-1$ elsewhere.
Since the parts of the external vertex boundary of $G$ in $\pi$ and $R(\pi)\setminus\pi$ are necessarily connected (by the maximality of $G$), by the partial Markov property above, we have $\mu^{\omega'}_G=\mu^{\tau}_G$, where $\tau$ equals $+1$ on $\pi$ and $-1$ elsewhere. Consequently, $R(T(\omega))$ is stochastically dominated by $\mu^{R(T(\tau))}_G$. However, Corollary~\ref{cor:CBC} also implies that $\mu^{\tau}_G$ stochastically dominates $\mu^{R(T(\tau))}_G$.
The lemma follows by taking a sequence of $\Lambda_n$ exhausting $\T$.

Lemma~3.1 of~\cite{GeoHig00} states that samples from every Gibbs measure have, almost surely, an \emph{infinite butterfly}, i.e., a pair of conjugate half-planes which contain infinite clusters of the same sign. The statement and proof of the lemma apply verbatim to our situation, making use of the modified Lemma 2.2.

Corollary~3.2 of~\cite{GeoHig00} is what we need, proving that every periodic Gibbs measure is a mixture of~$\mu^-$ and~$\mu^+$. Again, its statement and proof apply to our situation verbatim, finishing the argument.
\end{proof}

We remark that the proofs leading to the full characterization of Gibbs measures in~\cite{GeoHig00} apply to our situation with the exception of~Lemma~5.5 there. Adapting the latter to our model seems more delicate due to our weaker domain Markov property.

\begin{proof}[Proof of Theorem~\ref{thm:cluster-dichotomy}] Again, the analogy with the random-cluster model suggests that the proofs of \cite{DumSidTas13} apply in our context. Indeed the choice of $n\ge1$ and $x\le \tfrac1{\sqrt n}$ implies that the associated spin representation  enjoys the FKG inequality and the comparison between boundary conditions. It is in fact the case that the proofs of~\cite{DumSidTas13} apply here, with additional simplifications: one does not need to work both with the square lattice and its dual, and one can focus on the triangular lattice solely (since the duality here is simply flipping the spins). For this reason, we do not write out the proof. In order to illustrate one of the aspects of the argument though, we define the notion of symmetric domain and state an important lemma used repeatedly in the proof of \cite{DumSidTas13}.

\begin{figure}
\begin{center}
\includegraphics[scale=0.8]{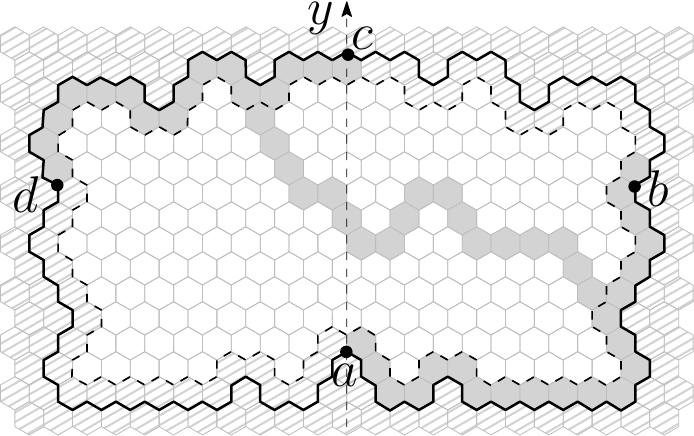}
\end{center}
\caption{A symmetric domain~$S$ (hexagons inside the dashed boundary) surrounded by a polygonal boundary~$P$ (bold boundary) with points~$a,b,c,d$ on it. The axis~$y$ is depicted in the middle. The boundary conditions are defined as follows: next to the arcs~$(ab)$ and~$(cd)$ the spins are~$1$ (marked with gray color) and the rest are~$-1$ (marked with dashed gray). Inside the domain the event of the crossing is depicted.}
\label{fig:domain_S}
\end{figure}

A {\em symmetric domain} $S$ (see Fig.~\ref{fig:domain_S}) is the collection of hexagons fully contained (all six edges) in the finite connected component of~$\bbH \setminus P$ for some self-avoiding polygon~$P$ in $\bbH$ which is symmetric with respect to the $y$-axis. Fix four points $a,b,c,d$ on $P$, with $b$ symmetric to $d$, and $a$ and $c$ the unique points on the $y$-axis. Define $(ab)$, $(bc)$, $(cd)$ and $(da)$ the arcs from $a$ to $b$, $b$ to $c$, $c$ to $d$ and $d$ to $a$ in $P$. Also, define the mixed boundary conditions to be made of pluses on hexagons bordering $(ab)$ or $(cd)$, and minuses everywhere else.
\begin{lemma}
Consider a symmetric domain $S$, then
\begin{equation}
\label{eq:sym-domains}
\mu_{S,n,x}^{\rm mix}[\exists \text{ a path of pluses from $(ab)$ to $(cd)$}]\ge\tfrac1{1+n}.
\end{equation}
\end{lemma}

\begin{proof}
The complement of the event that $(ab)$ is connected to $(cd)$ by a path of pluses is the event that $(bc)$ and $(da)$  are connected by a path of minuses. The symmetry between the pluses and minuses (note that the pluses may even have a slight advantage if there are hexagons of $(ab)$ or $(cd)$ intersecting the $y$-axis), and the fact that $(bc)$ and $(da)$ are in the same connected component of minuses outside of $S$ implies that the complement event has probability at most $n$ times the probability of our event. The proof follows readily.\end{proof}

Again, we highlight that this lemma is even more convenient than the corresponding claim in \cite{DumSidTas13}, since it does not involve the dual lattice. With this lemma at hand, the rest of the proof of \cite{DumSidTas13} is simple to adapt and we refer to the original article for details.
\end{proof}

\begin{remark}
Removing spins in all hexagons outside of~$P$ in the same way as in Proposition~\ref{prop:FKG} and a remark after it, one obtains~$1/2$ on the right-hand side of~\eqref{eq:sym-domains} using a complete symmetry of the pluses and minuses in the spin representation. We prefer keeping minus boundary conditions outside in order to be closer to the setup in~\cite{DumSidTas13}.
\end{remark}

We now show how to derive~Theorem~\ref{thm:loop-dichotomy} using Theorems~\ref{thm:cluster-infinite-volume} and~\ref{thm:cluster-dichotomy}.
Recall that $\Lambda_k$ is the ball of size $k$ around the origin, and denote $\partial\Lambda_k:=\Lambda_k\setminus\Lambda_{k-1}$.

\begin{proof}[Proof of Theorem~\ref{thm:loop-dichotomy}]
One simply defines $\bbP_{n,x}$ to be the pushforward of $\mu_{n,x}^+$ (or $\mu_{n,x}^-$) by the map $\sigma\mapsto \omega(\sigma)$. The convergence of finite-volume measures with empty boundary conditions follows directly from the corresponding statement for $\mu_{n,x}^+$.  The fact that configurations do not contain infinite paths follows from the fact that there is no coexistence of infinite connected components of pluses and minuses. Then, by~\cite[Lemma 3.7]{DumPelSam14}, $\bbP_{n,x}$ is a Gibbs measure, since it is obtained as a weak limit of finite-volume measures and is supported on configurations with no infinite paths.

We now show that~$\bbP_{n,x}$ is the unique periodic Gibbs measure for the loop~$O(n)$ model with edge-weight~$x$. Assume there exists a periodic Gibbs measure~$\bbP\neq\bbP_{n,x}$. Let $\omega$ be sampled from $\bbP$ and independently uniformly sample $S \in \{-1,1\}$. Let $\sigma$ be the unique spin configuration having $\omega(\sigma) = \omega$ and $\sigma_o=S$, where $o \in \bbT$ is the origin. Let $\mu$ be the law of $\sigma$. It is easy to see that~$\mu$ is a periodic Gibbs measure for the spin representation. Hence, Theorem~\ref{thm:cluster-infinite-volume} implies that $\mu$ can be written as a linear combination of~$\mu_{n,x}^+$ and~$\mu_{n,x}^-$. The pushforward of each of these measures is~$\bbP_{n,x}$. Thus, $\bbP=\bbP_{n,x}$, a contradiction.

In order to show the dichotomy, we use the alternative provided by Theorem~\ref{thm:cluster-dichotomy}. Fix $n\ge1$ and $x\le \tfrac1{\sqrt n}$. If none of the properties of Theorem~\ref{thm:cluster-dichotomy} are satisfied, then {\bf P4} is not satisfied and therefore there exists $c=c(n,x)>0$ such that
\[
\mu_{n,x}^-[a\longleftrightarrow\partial\Lambda_k(a)]\le \exp(-ck),
\]
for all $k\ge1$, where~$a\in\bbT$ and~$\Lambda_k(a)$ is the translation of~$\Lambda_k$ that maps~$0$ to~$a$. With the map $\omega\mapsto \sigma$, one easily sees that if the loop passing through a point~$a$ has diameter at least $k$, then there exists a path of pluses from one of the three hexagons bordering $a$, going to distance $k$ from $a$. Applying the previous displayed inequality to all points in~$\Lambda_k$, we obtain the first item of Theorem~\ref{thm:loop-dichotomy}.

If all the properties of Theorem~\ref{thm:cluster-dichotomy} are satisfied, we can prove that the second item of Theorem~\ref{thm:loop-dichotomy} is satisfied as follows. Fix $k$ and $\tau\in\{-1,1\}^\bbT$. Recall that $A_k$ is the set of edges of $\bbH$ belonging to a hexagon in $A'_k := \Lambda_{2k} \setminus \Lambda_k$. Set $B:=\Lambda_{3k/2}\setminus\Lambda_k$ and $B':=\Lambda_{2k}\setminus\Lambda_{3k/2}$. Let $\calE$ be the event that there exists a circuit of neighboring pluses in $B$ surrounding the origin. Similarly, let $\calF$ be the event that there exists a circuit of neighboring minuses in $B'$ surrounding the origin. Then, {\bf P5} (more precisely \eqref{eq:crossing} and the FKG inequality) implies that
$\mu_{B,n,x}^-[\calE]\ge c$
and
$\mu_{B',n,x}^+[\calF]\ge c.$
Then, conditioning on the values of the spins in~$B$ and using the domain Markov property, we obtain that
\begin{align*}
\mu_{A'_k,n,x}^\tau[\calE\cap \calF]
&= \sum_{\tau'\in \{-1,1\}^B}\1_{\tau'\in \calE}\cdot \mu_{A'_k,n,x}^\tau[\sigma_{|B} = \tau'] \cdot\mu_{A'_k,n,x}^\tau[\calF \, | \, \sigma_{|B} = \tau'] \\
&\ge \sum_{\tau'\in \{-1,1\}^B}\1_{\tau'\in \calE} \cdot \mu_{A'_k,n,x}^\tau[\sigma_{|B} = \tau']\cdot\mu_{B',n,x}^+[\calF]
\\&=  \mu_{B',n,x}^+[\calF]  \cdot \mu_{A'_k,n,x}^\tau[\calE] \ge c^2,
\end{align*}
where both inequalities are obtained using the comparison between the boundary conditions. Note that by writing~$\tau'\in \calE$ for~$\tau'\in \{-1,1\}^B$, we are slightly abusing the notation, since~$\calE$ is an event on~$\{-1,1\}^\bbH$. Nevertheless, as~$\calE$ is completely defined by the values of spins in~$B$, this does not lead to any ambiguity.

To conclude the proof, observe that on $\calE\cap \calF$, the configuration $\omega(\sigma)$ contains a loop which is contained in~$A_k$ and surrounds the origin, so that~\eqref{eq:RSW5} follows from Proposition~\ref{prop:bijection}. Finally, since {\bf P2} holds, the above argument showing uniqueness of the periodic Gibbs measure, implies uniqueness of the Gibbs measure (as every Gibbs measure for the spin representation lies stochastically between $\mu^-_{n,x}$ and $\mu^+_{n,x}$).
\end{proof}

\begin{proof}[Proof of Theorem~\ref{thm:cluster-h-critical}]
We may apply mutatis mutandis the existing arguments for showing that the critical point of random-cluster models on the square lattice is equal to the self-dual point. We even have several ways to proceed. Rather than using the original argument \cite{BefDum12}, we choose to use a recent short proof of this statement~\cite{DumRaoTas17}.

First, note that the choice of $n\ge1$ and $x\le \tfrac1{\sqrt n}$ guarantees that the associated spin representation satisfies the FKG lattice condition. Since it is also strictly positive by the finite energy property (each configuration in $\Sigma(G,\tau)$ has positive probability), we deduce by~\cite[Theorem (2.24)]{Gri06} that it is monotonic. A direct application of the result of \cite{DumRaoTas17} (with $e^h$ playing the role of $\tfrac p{1-p}$) thus implies the existence of $h_c\in\bbR$ such that
\begin{itemize}[noitemsep]
\item There exists $c>0$ such that for all $h\ge h_c$, $\mu^+_{n,x,h,0}[0\longleftrightarrow\infty]\ge c(h-h_c)$.
\item For $h<h_c$, there exists $c_h>0$ such that for any $k\ge1$,
$$\mu^+_{\Lambda_{2k},n,x,h,0}[0\longleftrightarrow \partial\Lambda_k]\le \exp(-c_hk).$$
\end{itemize}

We now prove that $h_c=0$ in two steps.
Consider the event $\calV_k$ that there exists a path of pluses in the trapeze $\{r+{\rm e}^{{\rm i}\pi/3}s:r,s\in\llbracket0,k\rrbracket\}$ from the top side to the bottom side. The complement of this event is the existence of a path of minuses from the left side to the right side so that, using the symmetry of the trapeze,
$$\mu_{n,x}^+[\calV_k]+\mu_{n,x}^-[\calV_k]=1.$$
 By the comparison between boundary conditions, we deduce that, for $h \ge 0$,
 $$\mu^+_{\Lambda_{2k},n,x,h,0}[0\longleftrightarrow \partial\Lambda_k]\ge \tfrac1k\cdot\mu_{n,x}^+[\calV_k]\ge \tfrac1{2k}.$$ This immediately implies that $h_c\le 0$ by item 2 above.

We now prove that $\mu_{n,x,h,0}^-[0\longleftrightarrow\infty]>0$ for any $h>h_c$. This property immediately implies that $h_c\ge0$, since otherwise there would be both infinite connected components of pluses and minuses for the measure $\mu_{n,x}^+$. To show that $\mu_{n,x,h,0}^-[0\leftrightarrow\infty]>0$, observe that the proof of \cite[Theorem~(4.63)]{Gri06} or \cite[Theorem~1.12]{Dum17a} applied to our context shows that for any fixed $n$ and~$x$, $\mu_{n,x,h,0}^+\ne\mu_{n,x,h,0}^-$ for at most countably many values of $h$. Therefore, there exists $h'\in(h_c,h)$ such that $\mu_{n,x,h',0}^+=\mu_{n,x,h',0}^-$ so that
\[
\mu_{n,x,h,0}^-[ 0\longleftrightarrow\infty]\ge \mu_{n,x,h',0}^-[ 0\longleftrightarrow\infty]=\mu_{n,x,h',0}^+[ 0\longleftrightarrow\infty]>0. \tag*{\qedhere}
\]
\end{proof}

\section{Proof of Theorem~\ref{thm:loop-macroscopic}}
\label{sec:theorem-big-loops}

The proof of Theorem~\ref{thm:loop-macroscopic} is a combination of several ingredients. We will work by contradiction, assuming that scenario~{\bf A1} of Theorem~\ref{thm:loop-dichotomy} is realized and all loops are small, and then proving that the probability of large loops is not exponentially small.
In order to do so, we will invoke so-called parafermionic observables to prove that weighted sums (defined below) of loop configurations with an additional path between two vertices on the boundary of a domain are not much smaller than weighted sums of loop configurations. Then, intuitively, the idea is to glue several domains together and combine these long paths into the large loop that we are looking for. The main problem here is that there can be loops exactly at the place of gluing. The solution is to use the fact that these loops are small by assumption, to condition on them, and, through the use of probabilistic estimates on relative weights of paths (see definition below), to show that long paths still exist with good probability and can be combined into a large loop. We start the proof by studying these relative weights in the next two sections.

In this section, we always assume that $n\ge1$ and $x\le \tfrac1{\sqrt n}$. We sometimes specify in addition that~$n\in [1,2]$ and that $x=x_c(n)$, which is always at most $\tfrac1{\sqrt n}$. To lighten the notation, we will drop $n$ and $x$ from the subscript in the measures or partition functions.

\subsection{Relative weight of a path}
\label{sec:rel-weight}

In this section, a finite subset of edges $\Omega$ of $\bbH$ is also seen as a subgraph of $\bbH$ with vertex-set given by the endpoints in $\Omega$. For a subset $A$ of vertices of $\Omega$, introduce the weighted sum $$Z_\Omega^A:=\sum_{\omega\in\calE(\Omega,A)}x^{|\omega|}n^{\ell(\omega)},$$
where $\calE(\Omega,A)$ is the set of subgraphs of $\Omega$ with even vertex degree for $v\notin A$ and vertex degree~1 for $v\in A$; as before, $|\omega|$ and~$\ell(\omega)$ denote the number of edges and loops in~$\omega$. Note that $Z^A_\Omega=0$ unless $|A|$ is even. When $A$ consists of two vertices $a$ and $b$, we write $Z^{a,b}_\Omega$ for $Z^{\{a,b\}}_\Omega$.
Define also the {\em relative weight} of a path $\gamma$ in $\Omega$ to be the following ratio:
\[
\wrel_\Omega(\gamma) = x^{|\gamma|}\cdot\frac{Z_{\Omega\setminus \gamma}^\emptyset}{Z_\Omega^\emptyset},
\]
where $\Omega \setminus \gamma$ is the subset of edges of $\bbH$ obtained from $\Omega$ by removing all the edges in $\gamma$ and the four additional edges incident to the endpoints of $\gamma$. We extend the above definition to the case when $\gamma$ is a subset of $\Omega$ consisting of disjoint paths, in which case $\Omega \setminus \gamma$ is obtained by removing all edges in $\gamma$ and the edges incident to the endpoints of the paths.

\begin{remark}
	When $n=1$ and vertices in $A$ are allowed to have degree 3, the sums and weights above are related via the Kramers--Wannier duality to spin correlations in the Ising model on $\bbH$. More precisely, the ratio of $Z_\Omega^A$ and $Z_\Omega^\emptyset$ is then simply the average of the random variable $\prod_{x\in A}\sigma_x$. In particular, it is always smaller than 1. The properties of $\wrel_\Omega(\gamma)$ are well-understood in this context, and are also related to the weights of the backbone in the random-current representation of the model \cite[page 353--355]{AizBarFer87}.
	In the following sections, we extend some of these properties to the regime $n\ge1$ and $n x^2\le 1$.
\end{remark}

Let us conclude this section by introducing notation. We write $\gamma:a\rightarrow b$ if $\gamma$ starts at $a$ and ends at $b$, and similarly, we write $\gamma: a \to B$ if $\gamma$ starts at $a$ and ends at some $b \in B$. We also write $\gamma\circ\eta$ for the concatenation of the paths $\gamma$ and $\eta$ (when $\eta$ starts at the end of $\gamma$). Note that by definition, the weights satisfy the {\em chain rule},
$$\wrel_{\Omega}(\gamma\circ\eta)=\wrel_\Omega(\gamma)\cdot\wrel_{\Omega\setminus\gamma}(\eta)=\wrel_{\Omega\setminus \eta}(\gamma)\cdot\wrel_\Omega(\eta).$$
Note also the simple relation for any vertices $a,b \in \Omega$:
\begin{equation}\label{eq:Z-and-relative-weights}
\frac{Z^{a,b}_{\Omega}}{Z^\emptyset_{\Omega}} = \sum_{\substack{\gamma\subset\Omega\\ \gamma:a\to b}}\wrel_\Omega(\gamma).
\end{equation}

\subsection{Probabilistic estimates on weights}

We will restrict ourselves to special subsets $\Omega$ of $\bbH$. We refer to Fig.~\ref{fig:triangle} for an illustration (there the case of a triangular domain is depicted). A subset $\Omega$ of edges of $\bbH$ is called a {\em domain} if there exists a self-avoiding polygon $P$ in $\bbH$ such that $\Omega$ is the set of edges with at least one endpoint in the finite connected component of $\bbH\setminus P$. Let $\partial\Omega$ be the set of vertices of $P$ neighboring a vertex in $\Omega$. Note that the vertices of $\partial\Omega$ are incident to exactly one edge of $\Omega$.

In the next two lemmas, we refer to sums of weights of configurations of the loop~$O(n)$ and its spin representation. We recall the notation and emphasize the difference: $Z_\Omega^A$ was defined in the previous subsection and refers to the loop~$O(n)$ model (note that it is different from~$Z_{\Omega,n,x}^\xi$ defined in the introduction), and~${\bf Z}_G^-$ refers to the spin representation and is defined by~\eqref{eq:dilute-Potts-measure}. We shall also use the notation~${\bf Z}^-_{G}[\cdot]:=\mu^-_{G,n,x}[\cdot]\cdot {\bf Z}^-_{G,n,x}$.

\begin{lemma}
	\label{lem_cutting_path_at_boundary}
	Fix $n\ge1$ and $x\le\tfrac1{\sqrt n}$. Then for any
	domain $\Omega$ and any $A \subset \partial\Omega$,
	$$\frac{Z^A_{\Omega}}{Z^\emptyset_{\Omega}} \le \frac{c_k}{n^{k/2}} \,,$$
	where~$k :=|A|/2$ and~$c_k := \tfrac{1}{k+1}{{2k}\choose{k}}$ is the $k$-th Catalan number.
\end{lemma}
\begin{proof}
Assume first that $k=1$ so that $A=\{a,b\}$ for some $a,b \in \partial \Omega$.
Let $P$ be the polygon defining the domain $\Omega$ and consider the set $G$ of hexagons having all their six edges in $\Omega\cup P$ (see Fig.~\ref{fig:domain_S}). Let $(ab)$ (resp.~$(ba)$) be the set of hexagons inside $P$ bordering the edges of $P$ contained in the arc between $a$ and $b$ when going counter-clockwise around $P$ (resp.~$b$ and~$a$).  Proposition~\ref{prop:bijection} describes a measure preserving bijection between the loop~$O(n)$ model and its spin representation. Moreover, the proof implies that the partition functions coincide, whence
\begin{align}Z_\Omega^\emptyset&={\bf Z}^-_{G}[\sigma_{|(ab)}=-,\sigma_{|(ba)}=-],\\
x^{m}n\cdot Z_\Omega^{a,b}&={\bf Z}^-_{G}[\sigma_{|(ab)}=+,\sigma_{|(ba)}=-],\\
x^{m'}n\cdot Z_\Omega^{a,b}&={\bf Z}^-_{G}[\sigma_{|(ab)}=-,\sigma_{|(ba)}=+],\\
x^{m+m'}n\cdot Z_\Omega^\emptyset&={\bf Z}^-_{G}[\sigma_{|(ab)}=+,\sigma_{|(ba)}=+],
\end{align}
where $m$ and $m'$ are the lengths of $P$-arcs between $a$ and $b$, and between $b$ and $a$. The additional~$x$ terms appear due to the fact that certain edges of~$P$ are separating hexagons bearing different spins and they are not counted in~$Z_\Omega^{a,b}$ and~$Z_\Omega^{\emptyset}$. The additional~$n$ terms appear because the exterior loop is not counted in~$Z_\Omega^{a,b}$ and~$Z_\Omega^{\emptyset}$.

Applying Corollary~\ref{cor:FKG-several-faces} for~$U=(ab)$ and~$V=(ba)$ gives
$$\mu_{G,n,x}^-[\sigma^{++}] \mu_{G,n,x}^-[\sigma^{--}] \ge \mu_{G,n,x}^-[\sigma^{+-}] \mu_{G,n,x}^-[\sigma^{-+}],$$
where $\sigma^{\eta\eta'}$ is the configuration coinciding with $\sigma$ except that it is equal to $\eta$ on $(ab)$ and $\eta'$ on $(ba)$.
Using the four displayed equalities above, we obtain
$$ (x^{m+m'}n\cdot Z_\Omega^\emptyset)\cdot (Z_\Omega^\emptyset) \ge (x^{m}n\cdot Z_\Omega^{a,b})\cdot (x^{m'}n\cdot Z_\Omega^{a,b}).$$
The term $x^{m+m'}n$ cancels out and we obtain
\begin{equation}\label{eq:cutting_path_at_boundary_two_points}
\frac{Z^{a,b}_{\Omega}}{Z^\emptyset_{\Omega}} \le \frac{1}{\sqrt{n}} .
\end{equation}

Assume now that $k \ge 2$.
Since $c_k$ counts the number of connectivity patterns on vertices of~$A$ induced by~$k$ (non-intersecting) paths linking them inside~$\Omega$, it suffices to show that, for any partition $\{a_1,b_1\}, \dots, \{a_k,b_k\}$ of $A$ arising from such a connectivity pattern,
\[ \sum_{\substack{\gamma_1,\dots,\gamma_k \subset \Omega\\\forall i~ \gamma_i \colon a_i \to b_i}}  \wrel_\Omega(\gamma_1 \cup \cdots \cup \gamma_k) \le \frac{1}{n^{k/2}} ,\]
where the sum is over collections $\{\gamma_1,\dots,\gamma_k\}$ of non-intersecting paths.
Yet, the chain rule gives
\[ \wrel_\Omega(\gamma_1 \cup \dots \cup \gamma_k) = \wrel_\Omega(\gamma_1) \cdot \wrel_{\Omega \setminus \gamma_1}(\gamma_2) \cdots \wrel_{\Omega \setminus (\gamma_1 \cup \cdots \cup \gamma_{k-1})}(\gamma_k) ,\]
so that the lemma follows by iteratively summing over $\gamma_k$ up to $\gamma_1$ and using~\eqref{eq:Z-and-relative-weights} and~\eqref{eq:cutting_path_at_boundary_two_points}, noting also that if $\Omega'\subset\Omega$ is obtained by removing a path from $\partial\Omega$ to itself, then each connected component of $\Omega'$ is also a domain.
\end{proof}

We now compare the relative weights of a path in different domains.
\begin{lemma}
\label{lem_monotonicity_domains}
Fix $n\ge1$ and $x\le\tfrac1{\sqrt n}$. Then for any two domains $\Omega\subset\Lambda$ and any path $\gamma\subset\Omega$,
$$\wrel_\Lambda(\gamma)\le 2\wrel_\Omega(\gamma).$$
Furthermore, if $\gamma$ starts and ends in $\partial\Omega\cap\partial\Lambda$, then $\wrel_\Lambda(\gamma)\le \wrel_\Omega(\gamma).$
\end{lemma}

\begin{proof}
We have
$$\frac{\wrel_{\Omega}(\gamma)}{\wrel_\Lambda(\gamma)}=\frac{Z^\emptyset_{\Omega\setminus\gamma}}{Z^\emptyset_{\Lambda\setminus\gamma}}\cdot \frac{Z^\emptyset_{\Lambda}}{Z^\emptyset_{\Omega}}.$$
Denote by~$\Omega^\bullet$ (resp.~$\Lambda^\bullet$) the set of hexagons fully contained in~$\Omega$ (resp.~$\Lambda$). Let $S$ be the set of hexagons having a vertex in common with $\gamma$, and denote $T:=\Lambda^\bullet\setminus\Omega^\bullet$. By Proposition~\ref{prop:bijection},
\begin{align*}
Z_\Lambda^\emptyset&={\bf Z}_{\Lambda^\bullet}^-,\\
Z_\Omega^\emptyset&={\bf Z}_{\Lambda^\bullet}^-[\sigma_{|T}=-],\\
Z_{\Omega\setminus\gamma}^\emptyset&={\bf Z}_{\Lambda^\bullet}^-[\sigma_{|T}=-,\sigma_{|S}=-]+{\bf Z}_{\Lambda^\bullet}^-[\sigma_{|T}=-,\sigma_{|S}=+]\ge {\bf Z}_{\Lambda^\bullet}^-[\sigma_{|T}=-,\sigma_{|S}=-].\end{align*}
Furthermore, the $\pm$ symmetry and the comparison between boundary conditions imply that $$\mu_{\Lambda^\bullet}^-[\sigma_{|S}=+]=\mu_{\Lambda^\bullet}^+[\sigma_{|S}=-]\le \mu_{\Lambda^\bullet}^-[\sigma_{|S}=-],$$ from which we deduce that
\begin{equation}\label{eq:ir}Z_{\Lambda\setminus\gamma}^\emptyset={\bf Z}_{\Lambda^\bullet}^-[\sigma_{|S}=-]+{\bf Z}_{\Lambda^\bullet}^-[\sigma_{|S}=+]\le 2{\bf Z}_{\Lambda^\bullet}^-[\sigma_{|S}=-].\end{equation}
 Overall, we have
$$\frac{Z^\emptyset_{\Omega\setminus\gamma}}{Z^\emptyset_{\Lambda\setminus\gamma}}\ge \tfrac12 \cdot \bbP_{\Lambda^\bullet}^-[\sigma_{|T}=- \mid \sigma_{|S}=-]\stackrel{\rm (FKG)}\ge \tfrac12 \cdot \bbP_{\Lambda^\bullet}^-[\sigma_{|T}=-]=\frac{Z_\Omega^\emptyset}{2Z_\Lambda^\emptyset}.$$
In the case where $\gamma$ starts and ends in $\partial\Omega\cap\partial\Omega'$, we have that
$Z_{\Lambda\setminus\gamma}^\emptyset={\bf Z}_{\Lambda^\bullet}^-[\sigma_{|S}=-]$ (the spins in $S$ cannot be equal to $+1$ since $S$ is touching the boundary), so that we do not lose the factor of $2$ in \eqref{eq:ir}.
\end{proof}

Let us mention an important (technical) consequence of the above lemmas (see Fig.~\ref{fig:cut_path}).
	
\begin{figure}
\begin{center}
\includegraphics[scale=0.8]{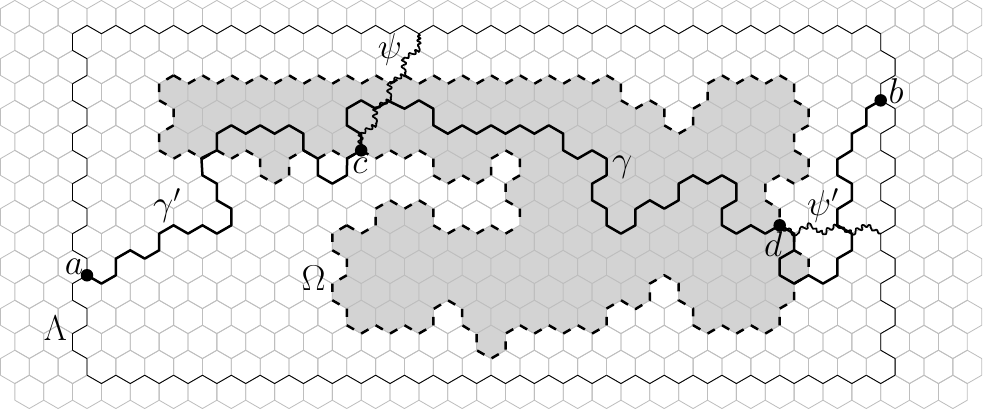}
\end{center}
\caption{The domain~$\Lambda$ and the subdomain~$\Omega$ (hexagons inside it are marked with gray). Points~$a,b$ are in~$\partial \Lambda$, points~$c,d$ are in~$\partial\Omega$, path~$\gamma\subset\Omega$ has endpoints~$c,d$, path~$\gamma'\subset\Lambda$ contains~$\gamma$ as a subpath and has endpoints~$a,b$. Path~$\gamma'$ can visit~$\Omega$ several times. Points~$c,d$ are connected to~$\partial\Lambda$ by paths~$\psi$ and~$\psi'$ (the wavy paths on the figure) of length at most~$k$. The paths~$\psi$ and~$\psi'$ can intersect~$\gamma$.
}
\label{fig:cut_path}
\end{figure}	
	
	\begin{corollary}
	\label{cor:cut}
Fix $n\ge1$ and $x\le\tfrac1{\sqrt n}$.  There exists a constant $C=C(n)>0$ such that the following holds. Consider two domains $\Omega\subset\Lambda$ with two boundary points $a,b\in\partial\Lambda$ and two points $c,d\in\partial\Omega$ at distance less than $k$ from $\partial\Lambda$ in $\Lambda$. Then for any path $\gamma$ in $\Omega$ from $c$ to $d$,
\[ \wrel_\Omega(\gamma) \ge e^{-Ck}
\sum_{\gamma'\in\Gamma'}\wrel_\Lambda(\gamma'),\]
where $\Gamma'$ is the set of paths in $\Lambda$ from $a$ to $b$ that contain $\gamma$ as a subpath.
\end{corollary}

\begin{proof}
Observe that the right-hand side of the inequality can be expressed as a sum over configurations in $\calE' := \bigcup_{\gamma' \in \Gamma'} \calE(\Lambda \setminus \gamma', \{a,b\})$.
Fix two paths $\psi$ and $\psi'$ in $\Lambda$ of length less than $k$, going from $\partial\Lambda$ to $c$ and $d$ respectively.
For $\omega \in \calE'$, define $\omega_1 := \omega \setminus (\gamma \cup \psi \cup \psi')$ and $\omega_2 := \omega \cap (\psi \cup \psi')$, and let $A$ be the set of degree 1 vertices in $\omega_1$ so that $\omega_1 \in \calE(\Lambda \setminus (\gamma \cup \psi \cup \psi'), A)$. Note that $A \subset \{a,b\} \cup V$, where $V$ is the set of endpoints of edges of $\omega_2$ in $\psi \cup \psi'$. Observe that $\Lambda \setminus (\gamma \cup \psi \cup \psi')$ is a union of domains with disjoint boundaries. Note also that $|V| \le 2k+2$ and that $\ell(\omega) \le \ell(\omega_1) + 2k$.
Since $\omega = \omega_1 \cup \omega_2 \cup \gamma$ for $\omega \in \calE'$, the map $\omega \mapsto (\omega_1,\omega_2)$ is injective on $\calE'$. Thus, summing over the choices of $\omega_1$, $\omega_2$ and $A$, and using Lemma~\ref{lem_cutting_path_at_boundary}, we obtain

\begin{align}
\sum_{\gamma'\in \Gamma'}\wrel_\Lambda(\gamma') = \frac{1}{Z^\emptyset_\Lambda} \sum_{\omega \in \calE'} x^{|\omega|} n^{\ell(\omega)}
 &\le \frac{x^{|\gamma|}}{Z^\emptyset_\Lambda} \cdot n^{2k} \cdot \sum_{\substack{A \subset \{a,b\} \cup V\\\omega_1 \in \calE(\Lambda \setminus (\gamma \cup \psi \cup \psi'), A)}} x^{|\omega_1|} n^{\ell(\omega_1)} \cdot \sum_{\omega_2 \subset \psi \cup \psi'} x^{|\omega_2|} \\
 &\le \frac{x^{|\gamma|}}{Z^\emptyset_\Lambda} \cdot n^{2k} (1+x)^{2k} \cdot \sum_{A \subset \{a,b\} \cup V} Z^A_{\Lambda \setminus (\gamma \cup \psi \cup \psi')} \\
 &\le \frac{x^{|\gamma|}}{Z^\emptyset_\Lambda} \cdot (2n)^{2k}\cdot \sum_{\ell = 0}^{k+2} \binom{2k+4}{2\ell} \frac{c_{\ell}}{n^{\ell/2}}\cdot Z_{\Lambda\setminus(\psi\cup\gamma\cup\psi')}^\emptyset\\
&\le (2n)^{2k}\cdot c_{k+2} \cdot 2^{2k+4}\cdot\wrel_{\Lambda}(\gamma),
\end{align}
where, in the last inequality, we used that $Z_{\Lambda\setminus\gamma}^\emptyset\ge Z_{\Lambda\setminus(\psi\cup\gamma\cup\psi')}^\emptyset$ to obtain the term $\wrel_\Lambda(\gamma)$.
We conclude the proof by noting that $\wrel_\Lambda(\gamma)\le 2\wrel_\Omega(\gamma)$ by Lemma~\ref{lem_monotonicity_domains} and that all the constant terms above are bounded by $\exp[O(k)]$.
\end{proof}

\subsection{The input from the parafermionic observable}

Fix $k$ even. Consider the equilateral triangular domain $\bbT_k$ of side length $k$ (see Fig.~\ref{fig:triangle}) defined as the set of edges of $\bbH$ with {\em at least one} endpoint in the subset
$\{0< y <\sqrt{3}(\tfrac{k}{2}-|x - \frac{k}{2}|)\}$ of~$\mathbb R^2$.
Let $\mathsf{B}_k$, $\mathsf{L}_k$ and $\mathsf{R}_k$ be the bottom, left and right parts of $\partial \bbT_k$.  Also, let $a$ be the point of cartesian coordinates $(\tfrac{k+1}{2},-\tfrac12)$ (it is in the middle of $\mathsf B_k$).

\begin{proposition}\label{prop_long_paths}
Fix $n \in [1,2]$ and $x=x_c(n)$. Then, for any even integer $k\ge1$,
$$\sum_{\substack{\gamma \subset \bbT_k\\ \gamma:a\to \mathsf{L}_k } } \wrel_{\bbT_k}(\gamma) \ge x^2.$$
\end{proposition}

\begin{figure}
\begin{center}
\includegraphics[scale=0.8]{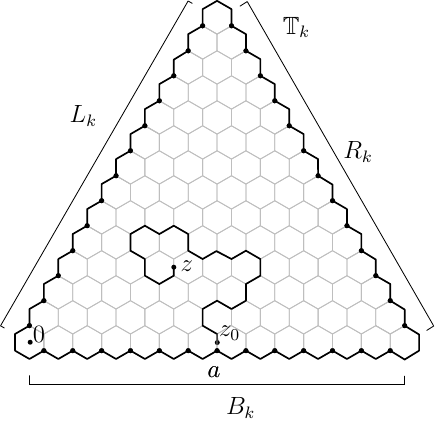}
\end{center}
\caption{Domain~$\Omega=\bbT_k$: the polygon~$P$ around it is in bold, the edges of~$\bbT_k$ are all those that lie inside~$P$ (note that the edges of $P$ are not in $\Omega$), boundary vertices~$\partial\bbT_k$ are marked with bullets, three different sides of~$\partial\bbT_k$ are denoted by~$\mathsf{B}_k$, $\mathsf{L}_k$, $\mathsf{R}_k$, each of them contains~$k$ hexagons, where~$k$ is an even number. The origin is located at the center of the leftmost hexagon on the bottom side, vertex~$a$ is in the middle of~$B_k$ and~$z_0$ is the midpoint of the edge in $\bbT_k$ emanating from~$a$. The path depicted on the picture is starting at~$z_0$ and ending at the midpoint~$z$ of an edge inside~$\bbT_k$ (as in the definition of the parafermionic observable). Furthermore, it has winding~$2\pi$ at $z$.}
\label{fig:triangle}
\end{figure}

\begin{proof}In order to prove this statement, we use the parafermionic observable. Set
$$\sigma = \sigma(n):=1 - \tfrac{3}{4\pi}\arccos(-n/2).$$
For this proof only, the paths $\gamma$ will be considered as going from the center $z_0$ of an edge to the center $z$ of another edge. Define $\Gamma_{z}=\Gamma_{z}(\Omega,z_0)$ for the set of paths in $\Omega$ from $z_0$ to $z$. For any $\gamma\in \Gamma_{z}$, $\wrel_\Omega(\gamma)$ is computed as in the case where $z_0$ and $z$ are vertices, and the notion of length~$|\gamma|$ is naturally extended by making the starting and ending half-edges contribute~$\frac{1}{2}$ instead of~1.

 Given a domain $\Omega$ and a center $z_0$ of an edge incident to $\partial\Omega$, define the parafermionic observable for any center $z$ of an edge in $\Omega$ as follows:
\[
F(z) :=\sum_{\gamma\in \Gamma_{z}} e^{-i\sigma\wind(\gamma)}\wrel_\Omega(\gamma),
\]
where $\wind(\gamma)$ is the total rotation when traversing $\gamma$ from~$z_0$ to~$z$.

It is by now classical (see \cite[Lemma 4]{Smi10a}) that $F$ satisfies the following relations when $x=x_c(n)$:
for the centers $p,q,r$ of the three edges incident to a vertex $v\in \Omega\setminus\partial\Omega$,
\begin{equation}
\label{eq_CR}
(p-v)F(p)+(q-v)F(q)+(r-v)F(r)=0,
\end{equation}
where $p-v$, $q-v$ or $r-v$ are seen as complex numbers.

We now focus on the domain $\bbT_k$ and $z_0=(\tfrac{k+1}2,0)$ (which is therefore the center of the edge of $\bbT_k$ incident to $a$). Summing the previous relation over all vertices $v\in \bbT_k\setminus\partial\bbT_k$, we find that the contributions of each inner edge to the relations around its endpoints cancel each other out, whence
\begin{equation}\label{eq:eq}
e^{-2\pi i/3}\sum_{z\in \overline{\mathsf L}_k} F(z)+e^{2\pi i/3}\sum_{z\in \overline{\mathsf R}_k}F(z)+\sum_{z\in \overline{\mathsf B}_k}F(z)=0,
\end{equation}
where $\overline{\mathsf L}_k$ (resp.~$\overline{\mathsf R}_k$ and $\overline{\mathsf B}_k$) denotes the set of centers of edges with one endpoint in $L_k$ (resp.~$R_k$ and $B_k$).

Now, if $z\in \overline{\mathsf L}_k\cup \overline{\mathsf R}_k\cup \overline{\mathsf B}_k$ then the observable can be computed simply using the observation that the winding of paths going from $z_0$ to $z$ is constant, i.e., does not depend on the path. More precisely, if $b$ is the vertex of $\partial\bbT_k$ associated to $z$ (recall that $a$ is associated to $z_0$), we obtain
\[
F(z)=\tfrac1{x}\cdot e^{-i\sigma w(z)}\sum_{\substack{\gamma\subset\bbT_k\\ \gamma:a\rightarrow b}}\wrel_{\bbT_k}(\gamma),
\]
where $w(z)$ is equal to $\pi/3$ on $\overline{\mathsf L}_k$, $-\pi/3$ on $\overline{\mathsf R}_k$, and $\pm\pi$ on $\overline{\mathsf B}_k$ depending on whether $z$ is on the left or right of $z_0$. Note that the term $\tfrac1x$ comes from the two missing half-edges necessary to complete $\gamma$ into a path from $a$ to $b$.
In particular, we obtain that
\[
e^{-2\pi i/3}\sum_{z\in \overline{\mathsf L}_k} F(z)+e^{2\pi i/3}\sum_{z\in \overline{\mathsf R}_k}F(z)= \frac1{x}\cdot2\cos((2+\sigma)\tfrac\pi3)\sum_{\substack{\gamma\subset \bbT_k\\ \gamma:a\rightarrow \mathsf L_k}}\wrel_{\bbT_k}(\gamma)=-\frac1{x^2}\sum_{\substack{\gamma\subset \bbT_k\\ \gamma:a\rightarrow \mathsf L_k}}\wrel_{\bbT_k}(\gamma),
\]
where we used that~$-\cos((2+\sigma)\tfrac\pi3) = \cos ((1-\sigma)\tfrac{\pi}{3}) = \tfrac{\sqrt{2+\sqrt{2-n}}}{2} = \tfrac1{2x_c} = \tfrac1{2x}$.

Since the empty walk is the only possible path from $z_0$ to $z_0$, we find $F(z_0)=1$. This, together with~$\sigma\le 1/2$, implies that
\[
\sum_{\substack{z\in \overline{\mathsf B}_k}}F(z)=F(z_0)+\tfrac1{x}\cdot\cos(\sigma\pi)\cdot  \hspace{-5mm}\sum_{\substack{\gamma\subset \bbT_k\\ \gamma:a\rightarrow \mathsf B_k\setminus\{a\}}}\wrel_{\bbT_k}(\gamma)\ge 1.
\]
Plugging this inequality and the previous displayed equation in \eqref{eq:eq} completes the proof.
\end{proof}

\subsection{Wrapping up the proof}
Fix $n\in[1,2]$ and $x=x_c(n)\le \tfrac1{\sqrt n}$. For convenience, we will write $Z_\Omega^A[\calE]$  for the weighted sum over configurations in $\calE\subset\calE(\Omega,A)$. Fix a large even integer $k$ and define
$$r:=\frac{k}{\log k}\qquad\text{and}\qquad\ell:=(\log k)^2.$$
We remark that the precise values of $r$ and $k$ are not important, we just need that $k/r$, $r/\ell$ and $\ell/\log k$ are sufficiently large.
For $1 \le s<k$, set $\bbT_{k,s}$ to be the domain $\bbT_{k-s}$ translated so that it is centered in the middle of $\bbT_k$.

Proposition~\ref{prop_long_paths} implies that
\begin{equation}\label{eq:ab}
\sum_{\substack{\gamma \subset \bbT_k\\ \gamma:a\to  \mathsf L_k} } \wrel_{\bbT_k}(\gamma) \ge x^2.
\end{equation}
We split the proof into two cases: either the paths $\gamma$ staying in $\bbT_k\setminus \bbT_{k,r}$ contribute at least half to the above sum, or the paths $\gamma$ intersecting $\bbT_{k,r}$ do. We will show that both of these cases are impossible when $k$ is large. We start with the case that the paths intersecting $\bbT_{k,r}$ contribute substantially, since this is from our point of view the most conceptual part of the argument.

\paragraph{Case 1.} Assume that (see Fig.~\ref{fig:case1})
\[
\sum_{\substack{\gamma \subset \bbT_k\\ \gamma:a\to  \mathsf L_k\\ \gamma\cap\bbT_{k,r}\ne\emptyset} } \wrel_{\bbT_k}(\gamma) \ge \frac{x^2}{2}\,.
\]
Since any path $\gamma'$ in $\bbT_k$ from $a$ to $\mathsf L_k$ intersecting $\bbT_{k,r}$ contains a subpath included in $\bbT_{k,\ell}$ also intersecting $\bbT_{k,r}$, there must exist $b\in\mathsf L_k$ and $c,d\in \partial \bbT_{k,\ell}$ satisfying
\begin{equation}\label{eq:abab}
\sum_{\substack{\gamma' \in\Gamma'_{bcd}}} \wrel_{\bbT_k}(\gamma') \ge \frac{x^2}{18k^3},
\end{equation}
where $\Gamma'_{bcd}$ is the set of paths $\gamma'$ in $\bbT_k$ from $a$ to $b$ containing a subpath in $\bbT_{k,\ell}$ from $c$ to $d$ intersecting $\bbT_{k,r}$. Note that we used that there are less than $k$ possibilities for $b$ and less than $3k$ possibilities for each of $c$ and $d$. In what follows, it will only be important whether~$c$ and~$d$ are on the same part or on different parts of~$\partial \bbT_{k,\ell}$. Using symmetry, we may assume that $c$ is on the bottom and that $d$ on the bottom or the left of $\partial\bbT_{k,\ell}$.

\begin{figure}
\centering
	\begin{subfigure}[t]{.48\textwidth}
		\includegraphics[scale=1.2,page=1]{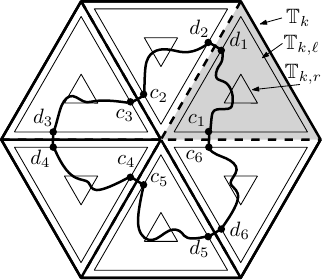}
		\caption{Point $c$ is on the bottom side of~$\bbT_k^\ell$ and~$d$ is on the left side of~$\bbT_k^\ell$.}
		\label{fig:case1-1}
	\end{subfigure}
	\begin{subfigure}{10pt}
		\quad
	\end{subfigure}
	\begin{subfigure}[t]{.48\textwidth}
		\includegraphics[scale=1.2,page=2]{case1.pdf}
		\caption{Both~$c$ and~$d$ are on the bottom side of~$\bbT_k^\ell$.}
		\label{fig:case1-2}
	\end{subfigure}
\caption{Case 1 of the proof. The triangle~$\bbT_k=\bbT_k^1$ is colored in gray. The triangles on both pictures represent~$\bbT_k^j$ (big triangles), $\bbT_{k,\ell}^j$ (middle triangles), $\bbT_{k,r}^j$ (smallest triangles). The set~$S$ is shown by dashed lines. Paths from~$c_j$ to~$d_j$ in~$\bbT_{k,\ell}^j$ and crossing~$\bbT_{k,r}^j$ have a big relative weight. Points of the set~$A=\{c_1,d_1,\dots,c_6,d_6\}$ are paired by short paths. The way vertices are paired by $\tau$ depends on the location of points~$c=c_1$ and~$d=d_1$.}
\label{fig:case1}
\end{figure}

Set $\bbT_k^1 = \bbT_k$, $c_1=c$ and $d_1=d$. Also, define $\bbT_k^{j+1}$, $c_{j+1}$ and $d_{j+1}$ to be the reflections of~$\bbT_k^j$, $c_j$ and $d_j$ with respect to $e^{j\pi i/3}\bbR$. Denote $\Lambda'_k := \bigcup_{j=1}^6 \bbT_k^j$ (this is the domain induced by the polygon surrounding $\Lambda_k$) and $A:=\{c_1,d_1,\dots,c_6,d_6\}$. We define~$\bbT_{k,s}^j$ similarly for $s \ge 1$ (in particular, for $s=r,\ell$). Let~$S$ be the set of edges of $\bbH$ belonging to the hexagons intersecting $\bbR\cup e^{i\pi/3}\bbR\cup e^{2i\pi/3}\bbR$.

For~$\omega\in\calE(\Lambda'_k,\emptyset)$, define~$\partial\omega$ to be the union of all loops of~$\omega$ that intersect~$S$. Let $\calE$ be the set of $\omega\in\calE(\Lambda'_k,\emptyset)$ that contain only loops of diameter less than~$\ell$. We will later use that the probability of $\calE$ is close to one if {\bf A1} of Theorem~\ref{thm:loop-dichotomy} holds.
Note that $\omega \setminus \partial \omega \subset \bbT_{k,1}^1 \cup \cdots \cup \bbT_{k,1}^6$ for all $\omega \in\calE(\Lambda'_k,\emptyset)$, and $\partial \omega \cap (\bbT_{k,\ell}^1 \cup \cdots \cup \bbT_{k,\ell}^6) = \emptyset$ for $\omega \in \calE$.
Now, for~$j=1,\dots,6$, let~$\mathrm{Int}^j(\omega)$ denote the connected component of the set~$\bbT_{k,1}^j \setminus \partial\omega$ that contains~$\bbT_{k,\ell}^j$. Note that by the definition of~$\calE$, the set~$\mathrm{Int}^j(\omega)$ is well-defined for any~$\omega\in\calE$. One may also check that $\mathrm{Int}^j(\omega)$ is in fact a domain. Define also~$\mathrm{Int}(\omega) := \mathrm{Int}^1(\omega) \cup \dots \cup \mathrm{Int}^6(\omega)$. We extend these definitions for configurations in $\calE(\Lambda'_k,A)$: we write $\calE^A$ for the set of $\omega\in\calE(\Lambda'_k,A)$ that contain only loops of diameter less than $\ell$ and paths which do not intersect $S$, and define $\mathrm{Int}(\omega)$ in an analogous way.

Consider $\Omega$ such that for some~$\omega\in\calE$ one has $\Omega = \mathrm{Int}(\omega)$, and denote $\Omega_j := \Omega \cap \bbT_k^j = \mathrm{Int}^j(\omega)$. Corollary~\ref{cor:cut} and \eqref{eq:abab} imply the existence of constants~$C,C'$ such that
\begin{equation}\label{eq:ok}
\forall j=1,\dots,6,\quad \sum_{\substack{\gamma \subset \bbT_{k,\ell}^j\\
\gamma:c_j\to d_j\\ \gamma\cap\bbT^j_{k,r}\ne \emptyset} } \wrel_{\Omega_j}(\gamma) \ge e^{-C'\ell}\sum_{\substack{\gamma' \in\Gamma'_{bcd}}} \wrel_{\bbT_k}(\gamma')  \ge e^{-C'\ell}\cdot \frac{x^2}{18k^3}\ge e^{-C\ell} \,.
\end{equation}
Denote by $\calF$ the set of configurations $\omega\in \calE^A$ which contain six paths, such that for all~$j=1,\dots,6$, one of these paths goes from $c_j$ to $d_j$ in $\bbT_{k,\ell}^j$ and intersects $\bbT_{k,r}^j$. Then, applying~\eqref{eq:ok} six times, we obtain
\begin{align*}
Z^A_{\Omega}\big[\calF \cap \calE(\Omega,A)\big]&\ge e^{-6C\ell}\  Z^\emptyset_\Omega\big[\calE \cap \calE(\Omega,\emptyset)\big].
\end{align*}
Now, we use that $\{\mathrm{Int} (\cdot )=\Omega\}$ is ``measurable from outside $\Omega$'', together with the domain Markov property of the loop model. More precisely, for any two configurations $\omega,\omega' \in \calE \cup \calE^A$ which coincide on $\Lambda'_k\setminus\Omega$, we have that $\mathrm{Int} (\omega )=\Omega$ if and only if $\mathrm{Int} (\omega')=\Omega$.
In addition, if $\omega \in \calE \cup \calE^A$ satisfies $\mathrm{Int} (\omega)=\Omega$, then it decomposes into two loop configurations $\omega \cap \Omega$ and $\omega \setminus \Omega$, the latter belonging to $\calE$. Using these observations, and denoting $\calE_\Omega := \{ \omega \setminus \Omega : \omega \in \calE,~\mathrm{Int} (\omega)=\Omega \}$, we obtain that
\begin{align*}
Z^A_{\Lambda'_k}[\{ \omega \in \calF : \mathrm{Int}(\omega)=\Omega \}]
 &= Z^\emptyset_{\Lambda'_k \setminus \Omega}[\calE_\Omega]\, Z^A_{\Omega}[\calF \cap \calE(\Omega,A)]\\
 &\ge e^{-6C\ell} Z^\emptyset_{\Lambda'_k \setminus \Omega}[\calE_\Omega]\, Z^\emptyset_{\Omega}[\calE \cap \calE(\Omega,\emptyset)]\\
 &=e^{-6C\ell}\  Z^\emptyset_{\Lambda'_k}[\{ \omega \in \calE : \mathrm{Int}(\omega)=\Omega \}].
\end{align*}
Summing over all $\Omega\in \{\mathrm{Int}(\omega) : \omega\in \calE\}$, we deduce that
\begin{align}
Z^A_{\Lambda'_k}[\calF]&\ge e^{-6C\ell}\ Z_{\Lambda'_k}^\emptyset[\calE].
\end{align}
We now wish to go back to configurations in $\calE(\Lambda'_k,\emptyset)$. Fix a collection~$\tau$ of six paths, each of length~$2\ell$, pairing the vertices of $A$ together in one of two following ways: if~$d$ is on the bottom side of $\bbT_{k,\ell}$, then we choose $\tau$ in such a way that the pairing is $(c_1,c_6)$, $(d_1,d_6)$, $(c_2,c_3)$, $(d_2,d_3)$, $(c_4,c_5)$, $(d_4,d_5)$; if~$d$ is on the left side of $\bbT_{k,\ell}$, then we consider a pairing $(d_1,d_2)$, $(c_2,c_3)$, $(d_3,d_4)$, $(c_4,c_5)$, $(d_5,d_6)$ and $(c_6,c_1)$. Let $\calG$ be the set of $\omega\in\calE(\Lambda'_k,\emptyset)$ containing a loop of diameter at least~$r-\ell$. Observe that $\omega\Delta\tau\in \calG$ as soon as $\omega\in\calF$. Moreover, $\omega \mapsto \omega\Delta \tau$ defines an injective map from~$\calF$ to~$\calG$ and the number of edges and loops in~$\omega\Delta\tau$ and~$\omega$ each differ by at most~$12\ell$, whence
\begin{align*}
Z^\emptyset_{\Lambda'_k}[\calG]&\ge (\tfrac xn)^{12\ell}\ Z^A_{\Lambda'_k}[\calF].
\end{align*}
Overall, using the two previous displayed inequalities and dividing by $Z^\emptyset_{\Lambda'_k}$ gives that
$$\bbP^\emptyset_{\Lambda'_k}[\calG]\ge  (\tfrac xn)^{12\ell}\,e^{-6C\ell}\ \bbP^\emptyset_{\Lambda'_k}[\calE].$$
Recall now the choice of~$r$ and $\ell$, and note that, if {\bf A1} of Theorem~\ref{thm:loop-dichotomy} is satisfied, then~$\bbP^\emptyset_{\Lambda'_k}[\calG]$ decays exponentially fast in $r$, and $\bbP^\emptyset_{\Lambda'_k}[\calE]$ tends to 1. This is contradictory for $k$ large.

\paragraph{Case 2.} Assume that (see Fig.~\ref{fig:case2})
$$\sum_{\substack{\gamma \subset \bbT_k\setminus \bbT_{k,r}\\ \gamma:a\to  \mathsf L_k} } \wrel_{\bbT_k}(\gamma) \ge \frac{x^2}{2}.$$
In this case, a path from $a$ to $\mathsf L_k$ staying in $\bbT_k\setminus\bbT_{k,r}$ must intersect the left or right boundary of the domain $\mathrm{Rect}_k$ enclosed in $[4r,k-4r]\times[0,4r]$. Thus, similarly to \eqref{eq:abab}, we get that there exist~$b\in L_k$ and~$d$ contained in the left or right boundary of~$\mathrm{Rect}_k$ such that
\begin{equation}
\sum_{\gamma' \in \Gamma'_{bd}} \wrel_{\bbT_k}(\gamma) \ge \frac{x^2}{4rk},
\end{equation}
where~$\Gamma'_{bd}$ is the set of paths~$\gamma'$ in $\bbT_k$ from~$a$ to~$b$ containing a subpath~$\gamma$ in $\mathrm{Rect}_k\setminus\bbT_{k,r}$ from~$a$ to~$d$. Here, we used that there are $k$ choices for $b$ and~$2r$ choices for~$d$. Below, we assume that~$d$ is contained in the left boundary of~$\mathrm{Rect}_k$, the case of the right boundary being completely analogous. In the same way as in the derivation of \eqref{eq:ok}, Corollary~\ref{cor:cut} implies that
\begin{equation}\label{eq:ababab}
\sum_{\substack{\gamma \subset \mathrm{Rect}_k\setminus\bbT_{k,r}\\
\gamma:a\to d} } \wrel_{\mathrm{Rect}_k}(\gamma) \ge  e^{-C' r} \sum_{\gamma' \in \Gamma'_{bd}} \wrel_{\bbT_k}(\gamma) \ge e^{-C r}.
\end{equation}
Consider $a_1$ and $d_1$, the reflections of $a$ and $d$ with respect to the horizontal line $\{(x,y)\in\bbR^2:y=2r\}$, and let $S$ be the set of edges of $\bbH$ belonging to the hexagons intersecting this line. Similarly to case 1, define $\calE$ to be set of $\omega\in\calE(\mathrm{Rect}_k,\emptyset)$ that contain only loops of diameter less than $\ell$, and for~$\omega\in\calE(\mathrm{Rect}_k,\emptyset)$, let~$\partial\omega$ be the union of all loops of $\omega$ intersecting~$S$. For $\omega \in \calE$, define~$\mathrm{Int}(\omega)\subset \mathrm{Rect}_k$ to be the union of the two connected components (each of which is a domain) in~$\mathrm{Rect}_k\setminus\partial\omega$ that contain the top and bottom sides of $\mathrm{Rect}_k$. Note that~$d$ is an endpoint of an edge in~$\mathrm{Int}(\omega)$, as the distance from~$d$ to~$S$ is at least~$r$.

\begin{figure}
\centering
	\begin{subfigure}[t]{.36\textwidth}
	\centering
		\includegraphics[scale=1.8]{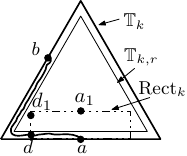}
		\caption{Paths from~$a$ to~$b$ contained in $\bbT_k\setminus \bbT_{k,r}$ have a large relative weight in~$\bbT_k$. The point~$d$ on the left (or right) side of rectangle~$\mathrm{Rect}_k$ is such that paths from~$a$ to~$d$ have a large relative weight in~$\mathrm{Rect}_k$.}
		\label{fig:case2-1}
	\end{subfigure}
	\begin{subfigure}{10pt}
		\quad
	\end{subfigure}
	\begin{subfigure}[t]{.6\textwidth}
	\centering
		\includegraphics[scale=0.7]{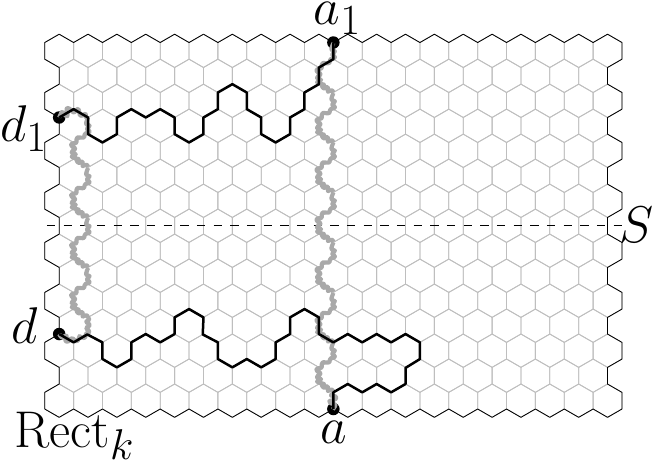}
		\caption{Here we zoom in on~$\mathrm{Rect}_k$. Points~$a_1,d_1$ are symmetric to~$a,d$ with respect to line~$S$. Points~$a,a_1$, as well as points~$d,d_1$, are linked by short straight paths (shown in gray) possibly intersecting paths~$a\to d$ and~$a_1\to d_1$. In any case, after removing the double edges these four paths create a big loop.}
		\label{fig:case2-2}
	\end{subfigure}
\caption{Case 2 of the proof.}
\label{fig:case2}
\end{figure}

By decomposing with respect to~$\mathrm{Int}(\omega)$ and using \eqref{eq:ababab} twice, we get
\[
Z^A_{\mathrm{Rect}_k}[\calF]\ge e^{-2Cr}Z_{\mathrm{Rect}_k}^\emptyset[\calE],
\]
where $A:=\{a,d,a_1,d_1\}$ and $\calF$ is the set of configurations $\omega\in\calE(\mathrm{Rect}_k,A)$ such that both paths do not intersect $S$ (hence, the one from~$a$ to~$d$ stays below~$S$, and the one from~$a_1$ to~$d_1$ stays above~$S$). Taking the symmetric difference with a configuration $\tau$ made of two paths, each of length at most $8r$, pairing $a$ to $a_1$, and $d$ to $d_1$, we obtain that
\[
Z^\emptyset_{\mathrm{Rect}_k}[\calG]\ge (\tfrac xn)^{16r}\ Z^A_{\mathrm{Rect}_k}[\calF],
\]
where $\calG$ is the set of configurations $\omega\in\calE(\mathrm{Rect}_k,\emptyset)$ containing a loop of diameter at least $k/2-20r$. Combining the two previous displayed inequalities gives
$$\bbP^\emptyset_{\mathrm{Rect}_k}[\calG]\ge (\tfrac xn)^{16r}e^{-2Cr}\ \bbP^\emptyset_{\mathrm{Rect}_k}[\calE].$$
We conclude as in case 1: if {\bf A1} of Theorem~\ref{thm:loop-dichotomy} is satisfied, $\bbP^\emptyset_{R_k}[\calG]$ decays exponentially fast in $k$, and $ \bbP^\emptyset_{\mathrm{Rect}_k}[\calE]$ tends to 1. This is contradictory for large $k$.
\bibliographystyle{abbrv}
\bibliography{biblicomplete2}

\begin{thebibliography}{10}

\bibitem{Aiz80}
M.~Aizenman.
\newblock Translation invariance and instability of phase coexistence in the
  two-dimensional {I}sing system.
\newblock {\em Comm. Math. Phys.}, 73(1):83--94, 1980.

\bibitem{AizBarFer87}
M.~Aizenman, D.~J. Barsky, and R.~Fern{{\'a}}ndez.
\newblock The phase \mbox{transition} in a general class of {I}sing-type models
  is sharp.
\newblock {\em J. Statist. Phys.}, 47(3-4):343--374, 1987.

\bibitem{Bax89}
R.~J. Baxter.
\newblock {\em Exactly solved models in statistical mechanics}.
\newblock Academic Press Inc. [Harcourt Brace Jovanovich Publishers], London,
  1989.
\newblock Reprint of the 1982 original.

\bibitem{BefDum12}
V.~Beffara and H.~{Duminil-Copin}.
\newblock The self-dual point of the two-dimensional random-cluster model is
  critical for {$q\geq 1$}.
\newblock {\em Probab. Theory Related Fields}, 153(3-4):511--542, 2012.

\bibitem{BenHon16}
S.~Benoist and C.~Hongler.
\newblock The scaling limit of critical {I}sing interfaces is {CLE}$_3$.
\newblock {\em The Annals of Probability}, 47(4):2049--2086, 2019.

\bibitem{Ber71}
V.~Berezinskii.
\newblock Destruction of long-range order in one-dimensional and
  two-dimensional systems having a continuous symmetry group {I}. classical
  systems.
\newblock {\em Sov. Phys. JETP}, 32(3):493--500, 1971.

\bibitem{BerLeC91}
D.~Bernard and A.~LeClair.
\newblock Quantum group symmetries and nonlocal currents in {$2$}{D} {QFT}.
\newblock {\em Comm. Math. Phys.}, 142(1):99--138, 1991.

\bibitem{BloNie89}
H.~W. Bl{\"o}te and B.~Nienhuis.
\newblock The phase diagram of the {${\rm O}(n)$} model.
\newblock {\em Physica A: Statistical Mechanics and its Applications},
  160(2):121 -- 134, 1989.

\bibitem{BurKea89}
R.~M. Burton and M.~Keane.
\newblock Density and uniqueness in percolation.
\newblock {\em Comm. Math. Phys.}, 121(3):501--505, 1989.

\bibitem{CamNew06}
F.~Camia and C.~M. Newman.
\newblock Two-dimensional critical percolation: the full scaling limit.
\newblock {\em Comm. Math. Phys.}, 268(1):1--38, 2006.

\bibitem{Car09}
J.~Cardy.
\newblock Discrete {H}olomorphicity at {T}wo-{D}imensional {C}ritical {P}oints.
\newblock {\em Journal of Statistical Physics}, 137:814--824, 2009.

\bibitem{Car10}
J.~Cardy.
\newblock Conformal field theory and statistical mechanics.
\newblock {\em Exact methods in low-dimensional statistical physics and quantum
  computing}, 89:65--98, 2010.

\bibitem{Cha98}
L.~Chayes.
\newblock Discontinuity of the spin-wave stiffness in the two-dimensional
  {$XY$} model.
\newblock {\em Comm. Math. Phys.}, 197(3):623--640, 1998.

\bibitem{CheDumHon14}
D.~Chelkak, H.~{Duminil-Copin}, C.~Hongler, A.~Kemppainen, and S.~Smirnov.
\newblock Convergence of {I}sing interfaces to {S}chramm's {SLE} curves.
\newblock {\em C. R. Acad. Sci. Paris Math.}, 352(2):157--161, 2014.

\bibitem{CheSmi12}
D.~Chelkak and S.~Smirnov.
\newblock Universality in the 2{D} {I}sing model and conformal invariance of
  fermionic observables.
\newblock {\em Invent. Math.}, 189(3):515--580, 2012.

\bibitem{DomMukNie81}
E.~Domany, D.~Mukamel, B.~Nienhuis, and A.~Schwimmer.
\newblock Duality relations and equivalences for models with {${\rm O}(N)$} and
  cubic symmetry.
\newblock {\em Nuclear Physics B}, 190(2):279--287, 1981.

\bibitem{Dum12}
H.~{Duminil-Copin}.
\newblock Divergence of the correlation length for critical planar {FK}
  percolation with $1\le q\le 4$ via parafermionic observables.
\newblock {\em Journal of Physics A: Mathematical and Theoretical},
  45(49):494013, 2012.

\bibitem{Dum17a}
H.~Duminil-Copin.
\newblock Lectures on the {I}sing and {P}otts models on the hypercubic lattice.
\newblock In {\em Random graphs, phase transitions, and the {G}aussian free
  field}, volume 304 of {\em Springer Proc. Math. Stat.}, pages 35--161.
  Springer, Cham, [2020] \copyright 2020.

\bibitem{DumKozYad11}
H.~{Duminil-Copin}, G.~Kozma, and A.~Yadin.
\newblock Supercritical self-avoiding walks are space-filling.
\newblock {\em Annales de l'Institut Henri Poincar\'e}, 50(2):315--326, 2015.

\bibitem{DumPelSam14}
H.~{Duminil-Copin}, R.~Peled, W.~Samotij, and Y.~Spinka.
\newblock Exponential decay of loop lengths in the loop {$O(n)$} model with
  large $n$.
\newblock {\em Communications in {M}athematical {P}hysics}, 349(3):777--817, 12
  2017.

\bibitem{DumRaoTas17}
H.~Duminil-Copin, A.~Raoufi, and V.~Tassion.
\newblock Sharp phase transition for the random-cluster and {P}otts models via
  decision trees.
\newblock {\em Annals of Mathematics}, 189(1):75--99, 2019.

\bibitem{DumSidTas13}
H.~{Duminil-Copin}, V.~Sidoravicius, and V.~Tassion.
\newblock Continuity of the phase transition for planar random-cluster and
  {P}otts models with $1\le q\le 4$.
\newblock {\em Communications in {M}athematical {P}hysics}, 349(1):47--107,
  2017.

\bibitem{DumSmi12a}
H.~{Duminil-Copin} and S.~Smirnov.
\newblock Conformal invariance of lattice models.
\newblock In {\em Probability and statistical physics in two and more
  dimensions}, volume~15 of {\em Clay Math. Proc.}, pages 213--276. Amer. Math.
  Soc., Providence, RI, 2012.

\bibitem{DumSmi12}
H.~{Duminil-Copin} and S.~Smirnov.
\newblock The connective constant of the honeycomb lattice equals
  {$\sqrt{2+\sqrt{2}}$}.
\newblock {\em Ann. of Math. (2)}, 175(3):1653--1665, 2012.

\bibitem{DumTas15c}
H.~Duminil-Copin and V.~Tassion.
\newblock {RSW} and box-crossing property for planar percolation.
\newblock IAMP proceedings, 2015.

\bibitem{FraKad80}
E.~Fradkin and L.~P. Kadanoff.
\newblock Disorder variables and para-fermions in two-dimensional statistical
  mechanics.
\newblock {\em Nuclear Physics B}, 170(1):1--15, 1980.

\bibitem{FroSpe81}
J.~Fr{{\"o}}hlich and T.~Spencer.
\newblock The {K}osterlitz-{T}houless transition in two-dimensional abelian
  spin systems and the {C}oulomb gas.
\newblock {\em Comm. Math. Phys.}, 81(4):527--602, 1981.

\bibitem{GeoHig00}
H.-O. Georgii and Y.~Higuchi.
\newblock Percolation and number of phases in the two-dimensional {I}sing
  model.
\newblock {\em J. Math. Phys.}, 41(3):1153--1169, 2000.
\newblock Probabilistic techniques in equilibrium and nonequilibrium
  statistical physics.

\bibitem{Gla13}
A.~Glazman.
\newblock Connective constant for a weighted self-avoiding walk on
  $\mathbb{Z}^2$.
\newblock {\em Electron. Commun. Probab.}, 20(86):1--13, 2015.

\bibitem{GlaMan18b}
A.~Glazman and I.~Manolescu.
\newblock Exponential decay in the loop~${O}(n)$ model:
  $n\geq1,x<\tfrac{1}{\sqrt{3}}+\varepsilon(n)$.
\newblock arXiv:1806.11302, 2018.

\bibitem{GlaMan18}
A.~Glazman and I.~Manolescu.
\newblock Uniform {L}ipschitz function on the triangular lattice have
  logarithmic variations.
\newblock arXiv:1806.05592, 2018.

\bibitem{Gri06}
G.~Grimmett.
\newblock {\em The random-cluster model}, volume 333 of {\em Grundlehren der
  Mathematischen Wissenschaften [Fundamental Principles of Mathematical
  Sciences]}.
\newblock Springer-Verlag, Berlin, 2006.

\bibitem{hammersley1971markov}
J.~M. Hammersley and P.~Clifford.
\newblock Markov fields on finite graphs and lattices.
\newblock {\em Unpublished manuscript}, 1971.

\bibitem{HelKra34}
G.~Heller and H.~Kramers.
\newblock Ein {K}lassisches {M}odell des {F}erromagnetikums und seine
  nachtr{\"a}gliche {Q}uantisierung im {G}ebiete tiefer {T}emperaturen.
\newblock {\em Ver. K. Ned. Akad. Wetensc.(Amsterdam)}, 37:378--385, 1934.

\bibitem{Hig81}
Y.~Higuchi.
\newblock On the absence of non-translation invariant {G}ibbs states for the
  two-dimensional {I}sing model.
\newblock In {\em Random fields, {V}ol. {I}, {II} ({E}sztergom, 1979)},
  volume~27 of {\em Colloq. Math. Soc. J{\'a}nos Bolyai}, pages 517--534.
  North-Holland, Amsterdam, 1981.

\bibitem{IkhCar09}
Y.~Ikhlef and J.~Cardy.
\newblock Discretely holomorphic parafermions and integrable loop models.
\newblock {\em J. Phys. A}, 42(10):102001, 11, 2009.

\bibitem{IkhWesWhe13}
Y.~Ikhlef, R.~Weston, M.~Wheeler, and P.~Zinn-Justin.
\newblock Discrete holomorphicity and quantized affine algebras.
\newblock {\em Journal of Physics A: Mathematical and Theoretical},
  46(26):265205, 2013.

\bibitem{KagNie04}
W.~Kager and B.~Nienhuis.
\newblock A guide to stochastic {L}{\"o}wner evolution and its applications.
\newblock {\em J. Statist. Phys.}, 115(5-6):1149--1229, 2004.

\bibitem{KosTho73}
J.~Kosterlitz and D.~Thouless.
\newblock Ordering, metastability and phase transitions in two-dimensional
  systems.
\newblock {\em Journal of Physics C: Solid State Physics}, 6(7):1181--1203,
  1973.

\bibitem{KosTho72}
J.~M. Kosterlitz and D.~Thouless.
\newblock Long range order and metastability in two dimensional solids and
  superfluids.({A}pplication of dislocation theory).
\newblock {\em Journal of Physics C: Solid State Physics}, 5(11):L124, 1972.

\bibitem{Len20}
W.~Lenz.
\newblock Beitrag zum {V}erst\"andnis der magnetischen {E}igenschaften in
  festen {K}\"orpern.
\newblock {\em Phys. Zeitschr.}, 21:613--615, 1920.

\bibitem{Nie82}
B.~Nienhuis.
\newblock Exact {C}ritical {P}oint and {C}ritical {E}xponents of
  $\mathrm{O}(n)$ {M}odels in {T}wo {D}imensions.
\newblock {\em Physical Review Letters}, 49(15):1062--1065, 1982.

\bibitem{Nie84}
B.~Nienhuis.
\newblock Coulomb gas description of {2D} critical behaviour.
\newblock {\em J. Statist. Phys.}, 34:731--761, 1984.

\bibitem{Nie91}
B.~Nienhuis.
\newblock Locus of the tricritical transition in a two-dimensional q-state
  potts model.
\newblock {\em Physica A: Statistical Mechanics and its Applications},
  177(1-3):109--113, 1991.

\bibitem{Nie10}
B.~Nienhuis.
\newblock Loop models.
\newblock {\em Exact Methods in Low-dimensional Statistical Physics and Quantum
  Computing}, 89:159--197, 2010.

\bibitem{Pas87}
V.~Pasquier.
\newblock Two-dimensional critical systems labelled by {D}ynkin diagrams.
\newblock {\em Nuclear Physics B}, 285:162--172, 1987.

\bibitem{PelSpi17}
R.~Peled and Y.~Spinka.
\newblock Lectures on the spin and loop ${O}(n)$ models.
\newblock In {\em Sojourns in Probability Theory and Statistical Physics-I},
  pages 246--320. Springer, 2019.

\bibitem{RajCar07}
M.~A. Rajabpour and J.~Cardy.
\newblock Discretely holomorphic parafermions in lattice {$Z_N$} models.
\newblock {\em J. Phys. A}, 40(49):14703--14713, 2007.

\bibitem{RivCar06}
V.~Riva and J.~Cardy.
\newblock Holomorphic parafermions in the {P}otts model and stochastic
  {L}oewner evolution.
\newblock {\em J. Stat. Mech. Theory Exp.}, (12):P12001, 19 pp. (electronic),
  2006.

\bibitem{Smi01}
S.~Smirnov.
\newblock Critical percolation in the plane: conformal invariance, {C}ardy's
  formula, scaling limits.
\newblock {\em C. R. Acad. Sci. Paris S{\'e}r. I Math.}, 333(3):239--244, 2001.

\bibitem{Smi06}
S.~Smirnov.
\newblock Towards conformal invariance of 2{D} lattice models.
\newblock In {\em International {C}ongress of {M}athematicians. {V}ol. {II}},
  pages 1421--1451. Eur. Math. Soc., Z{\"u}rich, 2006.

\bibitem{Smi10}
S.~Smirnov.
\newblock Conformal invariance in random cluster models. {I}. {H}olomorphic
  fermions in the {I}sing model.
\newblock {\em Ann. of Math. (2)}, 172(2):1435--1467, 2010.

\bibitem{Smi10a}
S.~Smirnov.
\newblock Discrete complex analysis and probability.
\newblock In {\em Proceedings of the {I}nternational {C}ongress of
  {M}athematicians. {V}olume {I}}, pages 595--621, New Delhi, 2010. Hindustan
  Book Agency.

\bibitem{Sta74}
E.~Stanley.
\newblock D-vector model or universality hamiltonian: properties of
  isotropically-interacting {D}-dimensional classical spins.
\newblock {\em Phase transition and critical phenomena}, 3:520, 1974.

\bibitem{Sta68}
H.~Stanley.
\newblock Dependence of critical properties on dimensionality of spins.
\newblock {\em Physical Review Letters}, 20(12):589--592, 1968.

\bibitem{Tas14b}
V.~Tassion.
\newblock Crossing probabilities for {V}oronoi percolation.
\newblock {\em The Annals of Probability}, 44(5):3385--3398, 2016.

\bibitem{VakLar66}
V.~Vaks and A.~Larkin.
\newblock On {P}hase {T}ransitions of {S}econd {O}rder.
\newblock {\em Soviet Journal of Experimental and Theoretical Physics}, 22:678,
  1966.

\end{thebibliography}

\end{document}